\documentclass[a4paper,11pt]{amsart}
\usepackage[utf8]{inputenc}
\usepackage{verbatim}
\usepackage{amssymb}
\usepackage{amsmath, amsthm, amsfonts}
\usepackage{color}

\newtheorem{theorem}{Theorem}[section]
\newtheorem{definition}[theorem]{Definition}
\newtheorem{prop}[theorem]{Proposition}
\newtheorem{lemma}[theorem]{Lemma}
\newtheorem{corollary}[theorem]{Corollary}
\newtheorem{remark}{Remark}

\newcommand{\vecz}{\mathbf{z}}
\newcommand{\vect}[1]{\mathbf{#1}}
\newcommand{\vecw}{\mathbf{w}}

\newcommand{\dm}{\mathrm dm_{n,p}^\hbar}
\newcommand{\dz}{\mathrm d\vecz}
\newcommand{\dzb}{\mathrm d\overline\vecz}
\newcommand{\E}{\mathcal E_{n,p}}
\newcommand{\T}{{\bf T}_{n,p}}
\newcommand{\To}{{\mathcal T}}

\newcommand{\Sn}{\mathbf{S}^n}
\newcommand{\vecx}{\mathbf{x}}
\newcommand{\vecy}{\mathbf{y}}
\newcommand{\vtheta}{\boldsymbol{\theta}}
\newcommand{\ds}{\mathrm d\mathbf S_n}
\newcommand{\bk}{{\bf k}}

\newcommand{\K}[1]{{{\bf K}_{n,#1}^\hbar}}
\newcommand{\Ku}{{{\bf K}_{n,-1}^\hbar}}

\newcommand{\B}{{{\mathfrak B}_{\hbar,p}}}
\newcommand{\Bemap}[1]{{{\mathbf B}^{\hbar}_{#1}}}

\newcommand{\Pro}{{{\mathbf P}}}
\newcommand{\Op}{\mathrm{Op}_\hbar^t}

\newcommand{\U}{{\mathbf U_{n,p}}}

\begin{document}
\title[Asymptotic expantion of covariant symbol on $\Sn$]
{Asymptotic expantion of covariant symbol on the complex unit sphere}
\author{Erik I. D\'{\i}az-Ort\'iz}
\address{CONACYT Research Fellow – Universidad Pedagógica Nacional - Unidad 201 Oaxaca}
\email{eidiazor@conacyt.mx}

\maketitle
\begin{abstract}
	Starting from a complete family (not defined by the reproducing kernel) for the unit sphere $\Sn$ in the complex $n$-space $\mathbb C^n$, we obtain an asymptotic expansion for the associated Berezin transform. The proof involves the computation of the asymptotic behaviour of functions in the complete family. Furthermore, we prove an Egorov-type theorem for the covariant symbol related to a pseudo-differential operator on $L^2(\Sn)$.
\end{abstract}

\section{Introduction and summary}
Let $\mathcal B_{\mathbb C^n}$ be the Bargmann spaces of all entire functions on $\mathbb C^n$ square integrable with respect to the Gaussian measure $\mathrm dv_n^\hbar (\vect z) = (\pi\hbar)^{-n} e^{-|\vect z|^2/\hbar}\mathrm d\vect z \mathrm d\overline {\vect z}$, $\hbar >0$, with $\vect z = (z_1 , \ldots , z_n)$, $|\vect z|^2 =|z_1|^2+\ldots + |z_n|^2$ and $\mathrm d\vect z \mathrm d\overline {\vect z}$ Lebesgue measure on $\mathbb C^n$.  

It is know that the Bargmann space $\mathcal B_{\mathbb C^n}$ enjoys the property of having a reproducing kernel $k_\vect w^\hbar(\vect z)=e^{\vect z\cdot \vect w/\hbar}$,  where $\vect z\cdot\vect w=z_1\overline w_1 + \cdots+z_n\overline w_n$ denotes the usual inner product in $\mathbb C^n$, and for all $f\in \mathcal B_{\mathbb C^n}$ the following
equation holds:
\begin{equation*}
f(\vect z)= \int_{\mathbb C^n} f(\vect w) k_\vect w^\hbar(\vect z) \mathrm dv_n^\hbar(\vect w)=\langle f, k_\vect z^\hbar \rangle.
\end{equation*}

Recall that for $f \in L^\infty(\mathbb C^n)$, the Berezin transform $\mathbf B^\hbar(f)$ of $f$ is the function on $\mathbb C^n$ defined by
\begin{equation*}
\mathbf B^\hbar(f)(\vect z)=\frac{\langle f k_\vect z^\hbar, k_\vect z^\hbar\rangle}{\langle k_\vect z^\hbar,k_\vect z^\hbar\rangle}= 
\big(k_\vect z(\vect z)\big)^{-1} \int_{\mathbb C^n} f(\vect w) |k_\vect z(\vect w)|^2\mathrm dv_n^\hbar(\vect w).
\end{equation*} 
Explicitly, $\mathbf B^\hbar (f)(\vect z)= (\pi \hbar)^{-n} \int f(\vect w) e^{-|\vect z -\vect w|^2/\hbar}\mathrm d \vect w \mathrm d \overline{ \vect w}$ which is just the standard formula for the solution at time $t = \hbar/4$ of the heat equation on $\mathbb C^n = \mathbb R^{2n}$ with initial data $f$. It follows that for $f \in L^\infty(\mathbb C^n)$ a smooth function in a neighbourhood of $\vect z$, its Berezin transform  has the following asymptotic expansion  when $h$ goes to zero
\begin{equation*}
\mathbf B^\hbar(f)(\vect z)= f(\vect z) + \hbar \frac{1}{1! 4}\partial_{\vect w \overline{\vect w}} + \hbar^2 \frac{1}{2! 4^2} \partial_{\vect w \overline{\vect w}} + \cdots, 
\end{equation*} 
with $\partial_{\vect w \overline{\vect w}}=\sum_{j=1}^n \partial^2/\partial w_j \partial \overline w_j$ denoting the Laplace operator.

It turns out that this kind of situation prevails in much greater generality. Namely, consider a domain $\Omega \in \mathbb C^n$ equiped with a K\"ahler form $\omega$, i.e. there must exist a strictly plurisubharmonic real-valued smooth function $\Phi$ such that $\omega=\partial \overline \partial \Phi$. For any $\hbar > 0$, we take the weighted Bergman spaces $\mathcal HL^2(\Omega, e^{-\Phi/\hbar}\mathrm d \mu)$ of all holomorphic functions in $L^2(\Omega,e^{-\Phi/\hbar}\mathrm d \mu)$, where $\mathrm d \mu(\vect z)= \mathrm{det}([g_{i\bar j}])\ \mathrm d\vect z \mathrm d\overline{\vect z}$ with $g_{i\bar j}=\partial^2 \Phi/\partial z_i \partial \overline z_j$. These spaces enjoys the property of having a reproducing kernel $K^\hbar(\vect z, \vect w)$. The corresponding Berezin transform is given by 
\begin{equation*}
\mathbf B^\hbar (f) (\vect z)=\int_\Omega f(\vect w) \frac{|K^\hbar(\vect z,\vect w)|^2}{K^\hbar(\vect z,\vect z)} e^{-\Phi(\vect w)/\hbar} \mathrm d\mu(\vect w).
\end{equation*}
Furthermore, for $f \in L^\infty(\Omega)$ one has the Toeplitz operator $\To_f$ with symbol $f$, namely, the operator on $\mathcal HL^2(\Omega, e^{-\Phi/\hbar}\mathrm d \mu)$ defined by $\To_f(\psi) = P_h( f\psi)$ , where $P_h: L^2(\Omega,e^{-\Psi/\hbar}\mathrm d\mu) \to \mathcal HL^2(\Omega, e^{-\Phi/\hbar}\mathrm d \mu)$ is the orthogonal projection.

Now, assume that $\Omega$ is a bounded symmetric domain and $e^\Psi$ is the Bergman kernel of $\Omega$; or that $\Omega$ is smoothly bounded and strictly pseudoconvex, and $e^{-\Psi}$ is a defining function for $\Omega^2$;  or that $\Omega= \mathbb C^n$ and $\Psi(\vect z ) = | \vect z |^2$. Then as $\hbar \searrow 0$, there are asymptotic expansions \cite{B74,E02,B-M-S94,C92} 
\begin{align}
K^\hbar(\vect z, \vect z)& = e^{\Phi(\vect z)/\hbar} \hbar^{-n} \sum_{\ell=0}^\infty \hbar^\ell b_\ell(\vect z);\label{asymptotic reproducing kernel}\\
\mathbf B^\hbar(f) & = \sum_{\ell=0}^\infty \hbar^\ell Q_\ell f \label{asymptotic covariant symbol}
\end{align}
for some functions $b_\ell \in C^\infty(\Omega)$, with $b_0=1$; some differential operators $Q_\ell$, with $Q_0$ the identity operator and $Q_1$ the Laplace–Beltrami operator with respect to the metric $g_{i\bar j}$.

On the other hand, in Ref. \cite{B74} Berezin defined the covariant symbol of a bounded linear operator $A$ on a Hilbert space $H$ endowed with the inner product $(\cdot,\cdot)$ as the complex-valued function $\mathbf B(A)$ on a set $M$ by
\begin{equation*}
\mathbf B(A)(\alpha)= \frac{(A \mathbf K(\cdot,\alpha),\mathbf K(\cdot,\alpha))}{(\mathbf K(\cdot,\alpha),\mathbf K(\cdot,\alpha))}, \quad \alpha \in M,
\end{equation*}
where the family $\{\mathbf K(\cdot,\alpha) \in H\;|\;\alpha \in M\}$ forms a complete system for $H$.

Under this definition, the Berezin transform makes sense not only considering the complete family $\{\mathbf K(\cdot, \alpha)\}$ as the reproducing kernel with one freezing variable, as in the weighted Bergman spaces, but for any complete family.  Hence, it is of interest to investigate if the spaces that have a complete family (not defined by the reproducing kernel) satisfy properties similar to (\ref{asymptotic reproducing kernel}) and (\ref{asymptotic covariant symbol}).

One such candidate, namely, the Hilbert space $\mathcal O$ of all functions in $L^2(\Sn)$ whose Poisson extension into the interior of
$\Sn$ is holomorphic, where $\Sn = \{\vect x \in \mathbb C^n |\; |x_1|^2 + \cdots + |x_n |^2 = 1\}$ and $L^2(\Sn)$ denotes the Hilbert space of square integrable functions with respect to the normalized surface measure
$\mathrm d\Sn(\vect x)$ on $\Sn$ and endowed with the usual inner product
\begin{equation}
\langle \phi,\psi\rangle_{\Sn}=\int_{\Sn} \phi(\vect x) \overline{\psi(\vect x)} \mathrm d \Sn(\vect x), \quad \phi,\psi \in L^2(\Sn).
\end{equation}

In Ref. \cite{D-18} it is defined a complete family $\mathcal K_p=\{\K{p}(\cdot,\vect z)| \vect z \in \mathbb C^n\}$ for $\mathcal O$   and the associated covariant symbol. 
The functions in  $\mathcal K_p$ are not obtained by the reproducing kernel and its definition is in terms of a suitable power series of the inner product $\vect x \cdot \vect z/\hbar$ which is an infinite series that is not in a closed form like an exponential function.



The aim of the present paper is to show an analogue of the asymptotic expansions (\ref{asymptotic reproducing kernel}) and (\ref{asymptotic covariant symbol}) for the functions $\K{p}(\cdot,\vect z)$, $\vect z \in \mathbb C^n$, and the associated Berezin transform, respectively.

The paper is organized as follows. Section \ref{section covariant symbol} is devoted to give a brief description and some properties of the family  $\mathcal K_p$ and the associated covariant symbol.

Analogues of the formula (\ref{asymptotic reproducing kernel}) for the asymptotic behaviour as $h  \searrow 0$ of a function in $\mathcal K_p$ is established in Section \ref{section semiclassical properties}. This formula is obtained following the work of Thomas and Wassell \cite{T-W95}. Moreover, to give rigorous proofs of subsequent theorems, we estimate the derivative of any order of the series defining $\K{p}(\vect x,\vect z)$.

Using the asymptotic expansion of functions in $\mathcal K_{p}$ and the stationary phase method, in Section \ref{section Asymptotic expansion of the Berezin transform} we get the asymptotic behaviour of the Berezin transform.

Moreover, since the covariant symbol is not only defined for Toeplitz operators but for any bounded linear operator, it is natural to ask if there exists an analogue of the asymptotic expansion obtained in Section \ref{section Asymptotic expansion of the Berezin transform} for the covariant symbol of more general operators than Toeplitz operators. For this reason, in  Section \ref{section Egorov-type theorem} we prove an Egorov-type theorem which relates the principal symbol of a given semiclassical pseudo-differential operator of order zero acting on $L^2(\Sn)$ (see Appendix \ref{appendix pseudo-differential operators} where we give a brief description of what we mean by semiclassical pseudo-differential operators on a manifold) with its covariant symbol in the semiclassical limit $\hbar \searrow 0$.


Throughout the paper, we will use the following  basic notation. For every $\vect z, \vect w \in \mathbb C^k$, $\vect z=(z_1,\ldots,z_k)$, $\vect w=(w_1,\ldots,w_k)$, and for every multi-index $\boldsymbol{ \ell}=(\ell_1,\ldots,\ell_k)\in \mathbb Z_+^k$ of length $k$, where $\mathbb Z_+$ is the set of nonnegative integers, let 
\begin{displaymath}
\vect z\cdot\vect w=\sum_{s=1}^k z_s\,\overline w_s, \hspace{.5cm} |\vect z|=\sqrt{\vect z \cdot \vect z},\hspace{.5cm}|\boldsymbol \ell|=\sum_{s=1}^k \ell_s, \hspace{.5cm} \boldsymbol{\ell} !=\prod_{s=1}^k \ell_s! , \hspace{.5cm} \vect z^{\boldsymbol \ell}=\prod_{s=1}^k z_s^{\ell_s}.
\end{displaymath}

Given $\omega\in \mathbb C$, let us denote its real and imaginary parts by $\Re(\omega)$ and $\Im(\omega)$ respectively.

Whenever convenient, we will abbreviate $\partial/\partial v_j, \partial/\partial \overline v_j$, etc., to $\partial_{v_j}, \partial_{\overline v_j}$, etc., respectively, $\partial_{v_1} \partial_{v_2}\ldots \partial_{v_k}$ to $\partial_{v_1v_2\cdots v_k}$, and $\partial_{\vect z}^{\vect \ell}=\partial_{z_1}^{\ell_1}\cdots \partial_{z_k}^{\ell_k}$ with $\partial_{z_j}^{\ell_j}=\underbrace{\partial_{z_j}\cdot\ldots\cdot{\partial_{z_j}}}_{{\ell_j}}$.

\section{The covariant symbol on $\mathcal O$} \label{section covariant symbol}
In this section we introduce the covariant symbol of a bounded linear operator with domain in $\mathcal O$ and some of properties it satisfies. In order to define this symbol, let us start by defining the functions $\K{p}(\cdot,\vect z)$, $\vect z \in \mathbb C^n$, that form a complete family $\mathcal K_p$ in $\mathcal O$ that is not obtained by the reproducing kernel. See Ref. \cite{D-18} for details.

\subsection{The functions $\K{p}$}\label{section coherent states}
 
Follows Ref. \cite{D-18}, let us consider the set of functions $\mathcal K_p=\{\K{p}(\cdot,\vect z)| \vect z \in \mathbb C^n\} \subset \mathcal O$, with
\begin{equation}\label{coherent states}
\K{p}(\vecx,\vecz):=\sum_{\ell=0}^\infty \frac{c_{\ell,p}}{\ell!}\left(\frac{\vecx \cdot \vecz}{\hbar}\right)^\ell, \quad (c_{\ell,p})^2= \frac{(n)_\ell}{(n+p)_\ell},\quad p>-n, \quad \vect x \in \Sn
\end{equation}
where $(a)_\ell$ stands for the Pochhammer symbol (raising factorial), $(a)_\ell=a (a+1)\cdots (a+\ell-1)=\frac{\Gamma(a+\ell)}{\Gamma(a)}$. 

In Ref. \cite{D-18} it is shown that the family $\mathcal K_p$ satisfies the conditions for defining  a Berezin symbolic calculus on $\mathcal O$, i.e. the family $\mathcal K_p$ satisfies the following two properties:

(\textbf I)	The family $\mathcal K_p$ forms a complete system for $\mathcal O$: for all $\Phi,\Psi \in \mathcal O$, Parseval's identity is valid
\begin{align*}
\langle \Phi,\Psi\rangle _{{\Sn}} & = \int_{\mathbb C^n} \langle \Phi,\K{p}(\cdot,\vecz)\rangle _{{\Sn}} \langle \K{p}(\cdot,\vecz),\Psi) \rangle _{{\Sn}}\dm(\vecz)\;,
\end{align*}
where 
\begin{equation}\label{measure on U}
\dm(\vecz)=\frac{1}{\Gamma(n+p)}\frac{2}{(\pi\hbar^2)^n}\left(\frac{|\vecz|}{\hbar}\right)^p \mathrm K_p\left(2\frac{|\vecz|}{\hbar}\right) \dz \dzb \;, \hspace{0.5cm} p>-n
\end{equation}
with $\dz\dzb$ denoting Lebesgue measure on $\mathbb C^n$ and $\Gamma$, $\mathrm K_\nu$ denoting the Gamma and MacDonald-Bessel function of order $\nu$, respectively (see Sections 8.4 and 8.5 of Ref. \cite{G94} for definition and expressions for these special functions).

(\textbf {II}) The map $\U: L^2(\Sn) \to L^2(\mathbb C^n, \dm)$ defined by 
\begin{equation}\label{definition operator U_{n,p}}
\U \Psi(\vect z)=\langle \Psi,\K{p}(\cdot,\vect z)\rangle_{\Sn}
\end{equation}
 is an embedding, where $L^2(\mathbb C^n , \dm)$ denotes the Hilbert space of square integrable functions on $\mathbb C^n$ with respect to the measure $\dm$ and endowed with the inner product
\begin{equation*}
(f,g)_p = \int_{\vecz \in \mathbb C^n} f(\vecz)\overline{g(\vecz)} \dm(\vecz)\;,\hspace{0.5cm} f,g \in L^2(\mathbb C^n , \dm).
\end{equation*}

In fact, in Ref. \cite{D-18} it is proved that the operator $\U$ is unitary onto the Hilbert space $\E$ of entire functions $f$ defined on $\mathbb C^n$ such that $||f||_p=\sqrt{(f,f)_p}$ is finite.

Moreover, for any $\vect z, \vect w \in \mathbb C^n$,
\begin{equation}\label{U_{n,p} of function K_{n,p}}
\U\left(\K{p}(\cdot,\vect w)\right)(\vect z)= \T(\vect z,\vect w),
\end{equation}
 where 
\begin{equation}\label{kernel}
\T(\vecz,\vecw)= \Gamma(n+p) \left(\frac{\vecz \cdot \vecw}{\hbar^2}\right)^{\frac{1}{2}(-p-n+1)} \mathrm I_{n+p-1} \left(\frac{2 \sqrt{\vecz \cdot \vecw}}{\hbar}\right),
\end{equation}
with $\mathrm I_k$ denoting the modified Bessel function of the first kind of order $k$ (see Secs. 8.4 and 8.5 of Ref. \cite{G94} for definition and expressions for this special function).

\subsection{The associated covariant symbol to the family $\mathcal K_p$}\label{section Berezin symbolic calculus}
Since the functions in $\mathcal K_p$ satisfy the properties (\textbf I) and (\textbf {II}),  
the Berezin's theory allow us to consider the following

\begin{definition}
	The covariant symbol $\B(A)$ of an operator $A$ with domain in $\mathcal O$ is defined as
	\begin{equation}\label{definition Berezin transform}
	\B( A)(\vecz)=\frac{\langle A\K{p}(\cdot,\vecz),\K{p}(\cdot,\vecz)\rangle_{{\Sn}}}{\langle\K{p}(\cdot,\vecz),\K{p}(\cdot,\vecz)\rangle_{{\Sn}}}\;,\quad \vecz \in \mathbb C^n.
	\end{equation}
\end{definition}

Note that this definition makes sense since the denominator is positive by  the relation $||\K{p}(\cdot,\vecz)||_{\Sn}^2=\T(\vect z,\vect z)$ (see Eqs. (\ref{definition operator U_{n,p}}) and (\ref{U_{n,p} of function K_{n,p}})) and Eq. (\ref{kernel}). Moreover, since the functions in $\mathcal K_p$ are continuous,  if $A:\mathcal O\to \mathcal O$ is an bounded operator, its covariant symbol can be extended uniquely to a function defined on a neighbourhood of the diagonal in $\mathbb C^n \times \mathbb C^n$  in such a way that it is holomorphic in the first factor and anti-holomorphic in the second. In fact, such an extension is given explicitly by
\begin{equation}\label{extended covariant symbol}
\B(A)(\vecw,\vecz):=\frac{\langle A\K{p}(\cdot,\vecz),\K{p}(\cdot,\vecw)\rangle_{{\Sn}}}{\langle\K{p}(\cdot,\vecz),\K{p}(\cdot,\vecw)\rangle_{{\Sn}}}.
\end{equation}

On the other hand, to every $\Phi \in C^\infty(\Sn)$, with $C^\infty(\Sn)$ denoting the algebra of complex-valued $C^\infty$ functions on $\Sn$, is associated a linear operator $\To_\Phi$ -the Toeplitz operator with symbol $\Phi$- that is defined for $\psi \in \mathcal O$ by
\begin{displaymath}
\To_\Phi (\psi)(\vect x)= \Pro \left(\Phi \psi\right)(\vect x), \quad \vect x \in \Sn
\end{displaymath}
where $\Pro:L^2(\Sn) \to \mathcal O$ is the orthogonal projection.

Starting from $\Phi\in C^\infty(\Sn)$, we can assign to it its Toeplitz operator $\To_{\Phi}$ and then assign to $\To_{\Phi}$ the covariant symbol $\B(\To_{\Phi})$. It is an element of $C^\infty(\mathbb C^n)$. Altogether we obtain a map $\Phi \mapsto \Bemap{p}(\Phi):=\B(\To_{\Phi})$.


\begin{definition}\label{definition Berezin transform}
	The map $\Bemap{p}:C^\infty (\Sn) \to C^\infty(\mathbb C^n)$ defined for $\Phi \in C^\infty (\Sn)$ by
	\begin{equation}\label{eq. Berezin transform}
	\Bemap{p} (\Phi)=\B(\To_{\Phi})
	\end{equation}
	is called Berezin transform.
\end{definition}

We end this section by showing some properties of the extended covariant symbol and the Berezin transform that we will use to obtain the asymptotic expansion of the covariant symbol.

\begin{prop}\label{extended covariant symbol T_U}
	Let $U \in \mathrm{SU}(n)$ (the group of $n \times n$ unitary matrices with unit determinant) and $\mathrm T_U:L^2(\Sn) \to L^2(\Sn)$ be the operator defined by $\mathrm T_U \Psi(\vect x)=\Psi(U^{-1}\vect x)$ with $\Psi \in L^2(\Sn)$. Let $A$ be a bounded linear operator with domain in $L^2(\Sn)$. Then:
	\begin{enumerate}
		\item For $ \vect w,\vect z \in \mathbb C^n$
		\begin{equation*}
		\B(A)(\vect w,\vect z)=\B(\mathrm T_{U^{-1}} A \mathrm T_U)(U^{-1}\vect w, U^{-1} \vect z),
		\end{equation*}
		In particular, $\B(A)=\mathrm T_{U}\B(\mathrm T_{U^{-1}} A \mathrm T_{U})$ if $\vect z=\vect w$.
		\item The Berezin transform $\Bemap{p}$ is invariant under the orthogonal transformations $\mathrm T_U$, i.e $\Bemap{p} \circ \mathrm T_U= \mathrm T_U\circ \Bemap{p}$.
	\end{enumerate}
	 \end{prop}
\begin{proof}
	Since the inner product in $\mathbb C^n$ is $\mathrm{SU}(n)$-invariant, we have $\K{p} (\cdot,\vect z)=\mathrm T_{U} \K{p}(\cdot,U^{-1}\vect z)$. 
	From the $\mathrm{SU}(n)$-invariance of  $\mathrm d \Sn$ and definitions of the extended covariant symbol and the Berezin transform (see Eq. (\ref{extended covariant symbol}) and Definition \ref{definition Berezin transform}) we  conclude the proof of Proposition \ref{extended covariant symbol T_U}.	
\end{proof}

\section{Semiclassical properties of the functions in $\mathcal K_p$} \label{section semiclassical properties}
Consider the set of functions $\mathcal K_p$ defined in subsection \ref{section coherent states}. The main goal of this section is to show semiclassical properties of the functions in $\mathcal K_p$, which in turn will allow us to obtain asymptotic expansions of the Berezin transform and the covariant symbol.

\subsection{Estimate of the inner product of functions in the complete family $\mathcal K_p$}
From Eqs.  (\ref{definition operator U_{n,p}}), (\ref{U_{n,p} of function K_{n,p}}), 	(\ref{kernel}) and the fact that the modified Bessel function $\mathrm I_\vartheta$, $\vartheta \in \mathbb R$, has the following asymptotic expression when $|\omega| \to \infty$ (see formula 8.451-5 of Ref. \cite{G94})
\begin{equation}\label{asymptotic expression modified Bessel function}
\mathrm I_\vartheta(\omega)=\frac{\mathrm e^\omega}{\sqrt{2\pi\omega}}\sum_{k=0}^\infty \frac{(-1)^k}{(2\omega)^k}\frac{\Gamma(\vartheta+k+\frac{1}{2})}{k!\Gamma(\vartheta-k+\frac{1}{2})}\;,\hspace{0.5cm} |\mathrm{Arg}(\omega)|<\frac{\pi}{2}\;.
\end{equation}
we can obtain the asymptotic expansion for the inner product of two functions in $\mathcal K_p$: 
\begin{prop}\label{proposition inner product coherent states}
	Let $n\ge 1$, $p>-n$ and $\vecz,\vecw \in \mathbb C^n$. Assume $\vect z \cdot \vect w \ne 0$ and $|\mathrm{Arg}(\vect z \cdot \vect w)| <  \pi$, then for $\hbar \to 0$
	\begin{align}
	\left\langle \K{p}(\cdot,\vect w),\K{p}(\cdot,\vect z)\right\rangle _{{\Sn}} & = \T(\vect z,\vect w) \\
	& = \frac{\Gamma(n+p)}{2\sqrt \pi} \left(\frac{\vect z \cdot \vect w}{\hbar^2}\right)^{\frac{1}{2}(-p-n+\frac{1}{2})} \mathrm{e}^{\frac{2}{\hbar}\sqrt{\vect z \cdot \vect w}}\biggl[1 -\nonumber\\
	& \quad\left.\frac{(n+p-\frac{3}{2})(n+p-\frac{1}{2})}{4\sqrt{\vect z\cdot\vect w}}\hbar+\mathrm O(\hbar^2)\right].\label{asymptotic inner product coherent states}
	\end{align}
	In particular, for $\vect z \in \mathbb C^n-\{\vect 0\}$:
	\begin{align}
	|| \K{p}(\cdot,\vecz)||_{\Sn}^2 & = \frac{\Gamma(n+p)}{2\sqrt \pi} \left(\frac{|\vecz|}{\hbar}\right)^{-p-n+\frac{1}{2}} \mathrm{e}^{\frac{2}{\hbar}|\vecz|} \biggl[1-\nonumber\\
	& \hspace{3cm}\frac{(n+p-\frac{3}{2})(n+p-\frac{1}{2})}{4|\vect z|}\hbar+\mathrm O(\hbar^2)\biggl].\label{norm coherent states}
	\end{align}
\end{prop}
\begin{remark} 
	We are mainly interested in using Proposition \ref{proposition inner product coherent states} for the cases $\vecz=\vecw$ (and then $\mathrm{Arg}(\vecz\cdot\vecw) = 0$) in this paper. The case when $| \mathrm{Arg}(\vecz\cdot\vecw)| = \pi$ requires the use of an asymptotic expression valid in a different region than the one we are considering in Proposition \ref{proposition inner product coherent states}. Thus if we take the branch of the square root function given by $\sqrt z = |z|^{1/2} \mathrm{exp}(\imath\mathrm{Arg}(z)/2)$ with $0 < \mathrm{Arg}(z) < 2\pi$ then by using formula 8.451-5 in Ref. \cite{G94} we obtain the asymptotic expression,
	\begin{align}
	\left\langle \K{p}(\cdot,\vect w),\K{p}(\cdot,\vect z)\right\rangle _{_{\Sn}}= & \frac{\Gamma(n+p)}{2\sqrt \pi} \left(\frac{\vect z \cdot \vect w}{\hbar^2}\right)^{\frac{1}{2}(-p-n+\frac{1}{2})} \nonumber\\
	& \left[\mathrm{e}^{\frac{2}{\hbar}\sqrt{\vect z \cdot \vect w}}  + \mathrm{e}^{-\frac{2}{\hbar}\sqrt{\vect z \cdot \vect w}}\mathrm{e}^{\pi\imath (n+p-\frac{1}{2})}\right] [1 + \mathrm O(\hbar)]\;.\label{other expression inner product}
	\end{align}
	Note that both asymptotic expressions in Eqs. (\ref{asymptotic inner product coherent states}) and (\ref{other expression inner product}) coincide up to an error of the order $\mathrm O(\hbar^\infty)$,  where $\mathrm O(\hbar^\infty)$ denotes a quantity tending to zero faster than any power of $\hbar$ , in the common region where they are valid.
\end{remark}

\subsection{Asymptotic of the functions in $\mathcal K_{-1}$} \label{asymptotic K{-1}}

In order to give a rigorous proof of Theorems \ref{theorem asymptotic Berezin transform p=1} and \ref{egorov theorem} we need to estimate the derivative of any order of the series defining $\K{p}(\vect x , \vect z)$. 
The definition of our functions $\K{p}(\vect x,\vect z)$ in terms of an infinite series (not in a closed form like an exponential function) might	look like it could be difficult to deal with them. However, from the work of Thomas and Wassell (see appendix B of Ref. \cite{T-W95}) we can obtain an asymptotic expansion of $\K{-1}(\vect x,\vect z)$, i.e. $n\ge 2$ and $p=-1$.

First, based on the definition of $\K{-1}$  (see Eq. (\ref{coherent states})) let us define, for $a \in \mathbb R-\{0\}$, the function $g_a : \mathbb C\to \mathbb C$ by
\begin{equation}\label{def g cap3}
g_a(z)= \sum_{\ell=0}^\infty \frac{\sqrt{a\ell+1}}{\ell !} \;z^\ell.
\end{equation}


Note that for $\vect z \in \mathbb C^n$, $\K{-1}(\vect x ,\vect z)=g_{\frac{1}{n-1}}\left(\frac{\vect x \cdot \vect z}{\hbar}\right)$, $\vect x \in \Sn$. 
In this subsection we obtain the main asymptotic term for the $s$th derivative of the function $g_a$ as a function of $z$ for $\Re(z) \to +\infty$ and $|\Im(z)|\le C \Re(z)$ with $C$ a positive constant, which in turn will allow us to obtain asymptotics of the $s$th derivative of the function $\K{-1}(\cdot,\vect z)$ for $\hbar$ small and $\vect x$ in either of the following regions on $\Sn$
\begin{equation}\label{regions W_z, V_z}
W_{\vect z}=\left\{\vect x \in \Sn\;|\; C\frac{\Re(\vect x\cdot\vect z)}{|\vect z|}\ge 1\right\}\;\;\mbox{and}\;\; V_{\vect z}=\Sn-W_{\vect z}\;.
\end{equation}

The proof of lemma below follows the work of Thomas and Wassel (see appendix B of Ref. \cite{T-W95} and lemma 10.1 of Ref. \cite{D-V09} for more details).
\begin{lemma}\label{asymptotic p=1}  
	Let $z \in \mathbb C$,  $s$ be any non-negative integer number and $a \in \mathbb R$ with $0 < a \le 2$. Suppose $\Re(z)>0$ and $|\Im(z)| \le C \Re(z)$ with $C$ a positive constant and $\Re(z)\to +\infty$. Then the $s\mathrm {th}$ derivative of  $g_a$ has the following asymptotic expansion:
	\begin{equation} \label{asymptotic expansion g_a}
	\frac{\mathrm d^sg_a(z)}{\mathrm dz^s}=\sqrt a z^{1/2}\exp(z)\left[1+\frac{a_{1,s}}{z}+\frac{a_{2,s}}{z^2}+...+\frac{a_{N,s}}{z^N} + \mathrm O(z^{-(N+1)})\right]
	\end{equation}
	with $a_{1,s},a_{2,s},...,a_{N,s}$ some constants.
\end{lemma}
\begin{proof}
	In the proof of this Lemma we denote by $\mathfrak K^{(\ell)}(z)$, $\ell=0,1,\ldots$, the $\ell$st derivative of a given function $\mathfrak K:\mathbb C \to \mathbb C$ evaluated in $z \in \mathbb C$, i.e. $\mathfrak K^{(\ell)}(z)=\frac{\mathrm d^\ell \mathfrak K(z)}{\mathrm d z^\ell}$.
	
	Using the fact $\frac{1}{\sqrt{a\ell+1}}=\frac{1}{\sqrt \pi}\int_{-\infty}^\infty e^{-(a\ell+1)t^2}\mathrm dt$  and the Taylor series for the exponential function we obtain
	
	\begin{align}
	g_a(z) & = \frac{1}{\sqrt \pi} \sum_{\ell=0}^\infty \frac{a\ell+1}{\ell!} z^\ell \int_{-\infty}^\infty e^{-(a\ell+1)t^2}\mathrm dt \nonumber\\
	& = \frac{2}{\sqrt\pi} \int_0^\infty \left[az e^{-at^2}+1\right]\exp(z e^{-at^2})e^{-t^2} \mathrm dt.\label{integral expression g_a}
	\end{align}
	
	Let us consider the change of variables $e^{-at^2}=1-w^2$ with $w \in[0,1]$. From Eq. (\ref{integral expression g_a})
	\begin{equation}
	g_a(z) = \frac{2}{\sqrt{ a\pi}} e^z\int_0^1 \left[az(1-w^2)+1\right]e^{-zw^2}\mathsf m(w)\mathrm d w \label{g}
	\end{equation}
	where
	\begin{displaymath}
	\mathsf m(w):= \frac{w(1-w^2)^{\frac{1}{a}-1}}{\sqrt{-\ln(1-w^2)}}.
	\end{displaymath}
	
	Let us write the integral in Eq. (\ref{g}) as an integral over the region $0\le t < 1/2$ plus an integral over the region $1/2\le t \le 1$ and let us call them $\mathrm J_a(\vect z)$ and $\mathrm I_a(\vect z)$ respectively. 
	
	
	Now we claim that for any $k \in \mathbb Z_+$, $\mathrm I_a^{(k)}(z)$ is $\mathrm O(z^{-\infty})$, where $\mathrm O(\hbar^\infty)$ denotes a quantity tending to zero faster than any power of $\hbar$ . In order to estimate $\mathrm I_a^{(k)} (\vect z)$, first write
	\begin{align*}
	\mathsf m(w)= \frac{w(1-w^2)^{\frac{1}{a}-\frac{1}{2}}}{ \left[-(1-w^2)\ln(1-w^2)\right]^{1/2}}\;.
	\end{align*}
	
	Since $0 < a \le 2$, then the numerator $(1-w^2)^{\frac{1}{a}-\frac{1}{2}}$ in the last expression is a bounded function in the interval $[1/2,1]$. Notice also that the function  $-1/\ln(1-w^2)$ is bounded from above by $-1/\ln(3/4)$ in the same interval. Since the integral $\int_{1/2}^1 1/\sqrt{1-w^2}\mathrm d w$ is finite and the absolute value of the derivative (respect to $z$) of any order of the integrand in Eq. (\ref{g}) is bounded by the function $(1-w^2)^{1/2}$, then by the dominated convergence theorem we have that for any natural number $r$
	\begin{align*}
	\left |z^r \mathrm I_a^{(k)}(z)
	\right| & \le M e^{-\Re(z)/4}|z|^{r+1}, 
	\end{align*}
	with $M$  a constant number independent of $z$. Since $|\Im (z)| \le C \Re(z)$, then $|z|\le C_1\Re(z)$ for some constant $C_1$. Therefore $\left |z^r\mathrm  I_a^{(k)}(z)\right|  \le M |\Re(z)|^{r+1}e^{-\Re(z)/4}$ which is a bounded function when $\Re(z)\to \infty$. Thus we have $\mathrm I_a^{(k)}(z)= \mathrm O(z^{-r})$ for all $r \in \mathbb N$, i.e.
	\begin{equation}\label{eq1 apendice c}
	\mathrm I_a^{(k)}(z)= \mathrm O(z^{-\infty})\;, \;\;\;\mbox{for}\;\;\; \Re(z)\to +\infty. 
	\end{equation}
	
	Let us now study the term $\mathrm J_a(z)$ and its $k$st derivative. First note that the function $\mathsf m$ has an even and $C^\infty$ extension to the interval $(-1/2,1/2)$ which implies that $\mathsf m$ has an asymptotic expansion in the interval $[0,1/2)$ in terms of even powers of the variable $w$ (see Ref. \cite{D-V09}). Namely, there exist numbers $b_{2j}$, $j=0,1,\ldots$ such that for $w \in [0,1/2)$
	\begin{displaymath}
	\mathsf m(w)- \sum_{j=0}^N b_{2j}w^{2j} = \mathrm O(w^{2N+2}) \; , \hspace{.5cm}b_0=\mathsf m(0)=1.
	\end{displaymath}
	
	Let us write  $\mathrm J_a^{(k)}(z) = \mathrm G_a^{(k)}(z) + \mathrm H_a^{(k)}(z)$
	with
	\begin{equation}\label{eq6 ec}
	\mathrm G_a(z)  = \int_0^{1/2}e^{-zw^2} \mathsf m(w) \mathrm dw ,\quad\mathrm H_a(z)  = \int_0^{1/2}a z(1-w^2)e^{-zw^2} \mathsf m(w) \mathrm dw.
	\end{equation}
	
	Let us study $\mathrm G_a^{(k)}$. Let $N$ be a natural number and define
	\begin{align}
	\mathrm G_{a,1}(z)  & = \int_0^{1/2}e^{-zw^2} \sum_{j=0}^N b_{2j}w^{2j}\mathrm dw,\label{definition G_{a,1}}\\
	\mathrm G_{a,2} (z) & = \int_0^{1/2}e^{-zw^2} \left(\mathsf m(w)-\sum_{j=0}^N b_{2j}w^{2j}\right)\mathrm dw.\label{definition G_{a,2}}
	\end{align}
	Notice that $\mathrm G_a^{(k)}(z)= \mathrm G_{a,1}^{(k)}(z)+\mathrm G_{a,2}^{(k)}(z)$. We claim that $\mathrm G_{a,2}^{(k)}(z)= \mathrm O(z^{-N-k-\frac{3}{2}})$. Since for any $\ell\in \mathbb Z_+$ the absolute value of the $\ell$st derivative of the integrand in Eq. (\ref{definition G_{a,2}}) is bounded by the function $w^{2N+2+2\ell}$, which in turn is integrable, then by the dominated convergence theorem we have 
	\begin{equation*}
	\left|\mathrm G_{a,2}^{(k)}(z)\right|  \le M \int_0^{1/2}e^{-\Re(z)w^2}w^{2N+2+2k}\mathrm dw \le \frac{M}{|z|^{(2N+2k+3)/2}},
	\end{equation*}
	with $M$ a constant only dependents on $N$ and where to obtain the last inequality we have considered the change of variables $\eta=\sqrt{\Re(z)}w$ and that $|\Im(z)|\le C \Re(z)$. 
	
	On the other hand, let us write $\mathrm G_{a,1}^{(k)}(z)$ as follows
	\begin{align}
	\mathrm G_{a,1}^{(k)} & = \sum_{j=0}^N b_{2j} \int_0^\infty e^{-zw^2}w^{2j}(-w^2)^k\mathrm dw- \sum_{j=0}^N b_{2j} \int_\frac{1}{2}^\infty
	e^{-zw^2}w^{2j}(-w^2)^k\mathrm dw\label{eq5 apendice c}.
	\end{align}
	
	The second term of $\mathrm G_{a,1}^{(k)}(z)$ is $\mathrm O(z^{-\infty})$ because for any natural number $r$ we have
	\begin{align}
	\left|z^r \sum_{j=0}^N b_{2j} \right.&  \left.\int_{\frac{1}{2}}^\infty e^{-zw^2}w^{2(k+j)} \mathrm dw\right| \nonumber\\
	&\le M |\Re(z)|^r \sum_{j=0}^N \frac{|b_{2j}|e^{-\Re(z)/4}}{|\Re(z)|^{k+j+1/2}} \int_0^\infty 
	e^{-t^2} (t+\sqrt{\Re(z)}/2)^{2(k+j)} \mathrm dt\label{help 1 G_{a,1} second term}\\
	& \le C_1 |\Re(z)|^{r}e^{-\Re(z)/4}\label{help 2 G_{a,1} second term}
	\end{align}
	with $M$, $C_1$ constants independent of $z$ and $C_1$ only depends on $N$, $r$ and $k$. To obtain the inequality (\ref{help 1 G_{a,1} second term}) we have considered the changes of variables $v=w-1/2$ and $t=\sqrt{\mathrm{Re}(z)}v$. The inequality (\ref{help 2 G_{a,1} second term}) is a consequence that the integrals $\int_0^\infty e^{-t^2} 
	(t+\sqrt{\mathrm{Re}(z)} /2)^{2(k+j)} \mathrm dt$ are polynomials in the variable $\sqrt{\Re(z)}$.

	On the other hand, since $\int_0^\infty e^{-zw^2}w^{2(k+j)}\mathrm d w=\frac{1}{z^{k+j}\sqrt z} \int_0^\infty e^{-t^2}t^{2(k+j)} \mathrm dt$, then from Eq. (\ref{eq5 apendice c})
	\begin{align}\label{eq6 apendice c}
	\mathrm G_{a,1}^{(k)}& = \sum_{j=0}^N (-1)^k  b_{2j}  \frac{d_{j+k}}{z^{k+j}\sqrt z} + \mathrm O(z^{-\infty})
	\end{align}
	with $\displaystyle d_{r}=\int_0^\infty e^{-t^2}t^{2r}\mathrm dt$. In particular $\displaystyle d_{0} =\frac{\sqrt \pi}{2}$.
	
	We conclude that
	\begin{align}
	\mathrm G_a^{(k)}(z)  = 
	(-1)^k \frac{\sqrt{z}} {z^{k+1}}\left [\sum_{j=0}^N\frac{d_{j+k}}{z^j} b_{2j}+\mathrm O(z^{-N-1})\right].\label{eq3 anr}
	\end{align}

	Let us now study $\mathrm H_a^{(k)}$, defined in Eq. (\ref{eq6 ec}). From the equality $\mathrm H_a(z)=az[\mathrm G_a(z)+\mathrm G_a^{(1)}(z)]$ and Eq. (\ref{eq3 anr})
	\begin{align}
	\mathrm H^{(k)}&  = a \left[z\mathrm G^{(k)}(z)+z\mathrm G^{(k+1)}(z)+k\mathrm G^{(k-1)}(z)+k\mathrm G^{(k)}(z)\right]\nonumber\\
	& =  a(-1)^{k}\frac{\sqrt z}{z^{k}}\biggl[b_0(d_k-kd_{k-1})\nonumber\\
	&\hspace{.3cm}+\sum_{j=1}^N\frac{1}{z^j}(b_{2j}-b_{2(j-1)})(d_{j+k}-kd_{j+k-1})+\mathrm O(z^{-N-1})\biggl]\label{eq7 ec}
	\end{align}
	where we are taking the convention that $ k d_ {k-1} = 0 $ when $ k = 0 $.
	
	Thus, from  Eqs. (\ref{g}), (\ref{eq1 apendice c}), (\ref{eq3 anr}) and (\ref{eq7 ec}) we obtain 
	\begin{align}
	g_a^{(s)}(z) & =  2 e^{z}\left(\frac{z}{a\pi}\right)^{\frac{1}{2}}\sum_{k=0}^s\binom{s}{k} \frac{(-1)^k}{z^k}\biggl[ a b_0(d_k-kd_{k-1})\nonumber\\
	& \hspace{0.5cm} + \sum_{j=1}^N\frac{1}{z^j}\biggl(a(b_{2j}-b_{2(j-1)})(d_{j+k}-kd_{j+k-1})\nonumber\\
	&\hspace{0.5cm} + b_{2(j-1)}d_{k+j-1}\biggl) + \mathrm O(z^{-N-1})\biggl].\label{previous g_a^s}
	\end{align}
	The asymptotic expansion of $g_a^{(s)}$ given in Eq. (\ref{asymptotic expansion g_a}) follows from Eq. (\ref{previous g_a^s}).
\end{proof}

Using Lemma \ref{asymptotic p=1} with $z = \vecx\cdot\vecz/\hbar$ we obtain the following asymptotic expansion of the $s$th derivative of $g_a$ evaluated at $z=\vect x \cdot \vect z/\hbar$ for $\hbar \to 0$:

\begin{prop}\label{proposition g_a^{s} evaluated}
	Let $\vecz \in \mathbb C^n-\{\vect 0\}$, $C$ be a constant greater than one, $W_\vect z$, $V_\vect z$ be the regions defined in Eq. (\ref{regions W_z, V_z})  and $s$ a non-negative integer number. Then for $\hbar \searrow 0$  we have
	\begin{align}
	g_a^{(s)}(\vect x \cdot\vect z/\hbar)& =\sqrt a \left[\frac{\vect x \cdot \vect z}{\hbar}\right]^{\frac{1}{2}} e^{\frac{\vect x \cdot \vect z}{\hbar}}[1 + \frac{a_{1,s}}{\vect x \cdot \vect z}\hbar + \mathrm O(\hbar^2)],&\mbox{for}\;\; \vect x \in W_\vect z.\label{g_a^{s} evaluated in W_z}\\[0.2cm]
	|g_a^{(s)}(\vect x \cdot\vect z/\hbar)| & \le \frac{C_1}{\hbar} e^{|\vect z|/\hbar} e^{\mu|\vect z|/\hbar} \left(\frac{|\vect z|}{\hbar}+1\right)
	,&\mbox{for}\;\; \vect x \in V_\vect z,\label{g_a^{s} evaluated in V_z}
	\end{align}
	with $C_1$ a constant and $\mu=\frac{1}{C}-1<0$. 
	In particular, for $\vect x \in W_\vect z$
	\begin{equation}\label{asymptotic function K{-1}}
	\K{-1}(\vecx,\vecz)=\left[\frac{\vecx\cdot\vecz}{\hbar(n-1)}\right]^{\frac{1}{2}} \mathrm{exp}\left(\frac{\vecx\cdot\vecz}{\hbar}\right)\left[ 1+\frac{a_1}{\vecx\cdot\vecz} \hbar+\mathrm O(\hbar^2)\right]\;.
	\end{equation}
	with $a_1=a_{1,0}=\frac{1}{2}(n-\frac{5}{4})$.
\end{prop}
\begin{proof}
	
	First we note that for $\vect x \in W_\vect z$,  $|\Im(\vecx\cdot\vecz)| \le |\vecz|\le C\Re(\vecx\cdot\vecz)$, therefore we can use the asymptotic expression of the function $g_a^{(s)}$ given in Lemma \ref{asymptotic p=1}. Thus, from Eq. (\ref{asymptotic expansion g_a})
	\begin{equation}
	g_a^{(s)}(\vect x \cdot\vect z/\hbar)=\sqrt a \left[\frac{\vect x \cdot \vect z}{\hbar}\right]^{\frac{1}{2}} \mathrm{exp}\left(\frac{\vect x \cdot \vect z}{\hbar}\right)[1 + \frac{a_{1,s}}{\vect x \cdot \vect z}\hbar + \mathrm E_s(\vect x ,\hbar)],
	\end{equation}	
	where the error term $\mathrm E_s(\vect x ,\hbar)$ is $\mathrm O(\hbar^2)$ uniformly with respect to $\vect x \in W_\vect z$ because in such a region we have $1/|\vect x \cdot \vect z|\le C/|\vect z|$.
	
	To show Eq. (\ref{g_a^{s} evaluated in V_z}), let us consider the integral expansion (\ref{g}) to find that the $s$th derivative of $g$ evaluated at $z=\vect x \cdot \vect z/\hbar$ can be written as
	\begin{align}
	g_a^{(s)}\left(\frac{\vect x \cdot \vect z}{\hbar}\right) & =\int_0^1 \left[P_{s,1}(w) \frac{\vect x \cdot \vect z}{\hbar}+ P_{s,2}(w)\right] e^{\frac{\vect x \cdot \vect z}{\hbar}(1-w^2)}\mathsf m(w)\mathrm dw\nonumber\\
	\intertext{with $P_{s,1}$ and $P_{s,2}$ polynomial functions on the variable $w$. Thus for $\vect x \in V_\vect z$}
	\left|g_a^{(s)}\left(\frac{\vect x \cdot \vect z}{\hbar}\right)\right| & \le e^{|\vect z|/\hbar}\int_0^1 \left|P_{s,1}(w) \frac{\vect x \cdot \vect z}{\hbar}+ P_{s,2}(w)\right|
	\left|\mathsf m(w)\right|\nonumber\\
	&\hspace{1.5cm}\exp\left(\frac{|\vect z|}{\hbar} \left[\frac{\Re(\vect x\cdot\vect z)}{|\vect z|}
	(1-w^2)-1\right]\right)\mathrm dw\nonumber\\
	& \le \frac{C_1}{\hbar} e^{|\vect z|/\hbar} e^{\mu|\vect z|/\hbar}\left(\frac{|\vect z|}{\hbar}+1\right)\nonumber
	\end{align}
	with $C_1$ a constant and $\mu=\frac{1}{C}-1<0$. 
\end{proof}

\section{Asymptotic expansion of the Berezin transform.}\label{section Asymptotic expansion of the Berezin transform}
 
In this section we show an analogue of the asymptotic expansion (\ref{asymptotic covariant symbol}) for the associated Berezin transform to the family $\mathcal K_p$ (see Eq. (\ref{eq. Berezin transform})).

It is possible to obtain all the asymptotic expansion of the Berezin transform $\Bemap{p}(\phi)$, $\phi\in C^\infty(\Sn)$, $p>-n$, when $\phi$ is a polynomial. For the general case, when $\phi$ is a smooth function, we consider only the case $n\ge 2$ and $p=0$ or $p=-1$.


\begin{theorem}\label{asymptotic berezin transform polynomial}
Let  $\bk\in \mathbb Z_+^n$ be a multi-index and $\varphi_\vect k(\vect x)= \vect x^\vect k$, with $\vect x \in \Sn$. Then for $\vect w, \vect z \in \mathbb C^n-\{\vect 0\}$
\begin{align*}
\B(\To_{\varphi_\bk})(\vecw,\vecz) & = \frac{\left(\displaystyle\frac{\vecw\cdot\vecz}{\hbar^2}\right)^{\frac{1}{2}(n+p-1)}}{\mathrm I_{n+p-1}\left(2\sqrt{\vecw\cdot\vecz}/\hbar\right)}\left(\frac{\vecw}{\hbar}\right)^\bk  \sum_{\ell=0}^\infty \left(\frac{\vecw\cdot\vecz}{\hbar^2}\right)^\ell \frac{c_{\ell,p}}{c_{|\bk|+\ell,p}}\\
&\hspace{0.5cm} \frac{1}{\ell!\Gamma(|\bk|+\ell+n+p)}
\end{align*}
where $c_{\ell,p}$ was define in Eq. (\ref{coherent states}).
\end{theorem}
\begin{proof} From Eq. (\ref{extended covariant symbol}) and properties of the orthogonal projection
\begin{align*}
\B(\To_{\varphi_\bk})(\vecw,\vecz) = \frac{\langle \vecx^{\bk}\K{p}(\cdot,\vecz),\K{p}(\cdot,\vecw)\rangle_{{\Sn}}}{\langle\K{p}(\cdot,\vecz),\K{p}(\cdot,\vecw)\rangle_{{\Sn}}}\;.
\end{align*}
From the dominated convergence theorem, Lemma A1 in Ref. \cite{D-18} and  Eqs. (\ref{definition operator U_{n,p}}), (\ref{U_{n,p} of function K_{n,p}}) and 	(\ref{kernel}) we conclude the proof of Theorem \ref{asymptotic berezin transform polynomial}.
\end{proof}

\begin{corollary}
	Let $\vect z \in \mathbb C^n-\{\vect 0\}$, $\vect k \in \mathbb Z_+^n$ be a multi-index and $\varphi_\vect k$ as above. Then for $\hbar \searrow 0$
	\begin{equation}\label{asymptotic berezin transform polynomial p=0}
	\Bemap{0}(\varphi_\vect k)(\vect z)= \left(\frac{\vect z}{|\vect z|} \right)^{\vect k} \left[1-\frac{|\vect k|(|\vect k|+2n-2)}{4|\vect z|}\hbar + \mathrm O(\hbar^2)\right].
	\end{equation}
\end{corollary}
\begin{proof}
	From Theorem \ref{asymptotic berezin transform polynomial} and the expression of the modified Bessel function $\mathrm I_\nu$, $\nu \in \mathbb R$,  as a power series (see formula 8.445 of Ref. \cite{G94}) we have
	\begin{align*}
	\Bemap{0}(\varphi_{\vect k})(\vect z) & = \left(\frac{\vect z}{|\vect z|}\right)^\bk \frac{\mathrm I_{n+|\bk|-1}\left(2|\vect z|/\hbar\right)}{\mathrm I_{n-1}\left(2|\vect z|/\hbar\right)}.
	\end{align*}
	
	Using the asymptotic expression of the modified Bessel function $\mathrm I_\nu(\omega)$ when $|\omega| \to \infty$ (see Eq. (\ref{asymptotic expression modified Bessel function})
	) we obtain the asymptotic expression given in Eq. (\ref{asymptotic berezin transform polynomial p=0}).
\end{proof}


\begin{theorem}\label{theorem asymptotic Berezin transform}
Let $\vect z \in \mathbb C^n-\{\mathbf 0\}$ and $\phi$ be a smooth function on $\Sn$. Then for $\hbar \searrow 0$
\begin{align}
\Bemap{0}(\phi)(\vect z)& =\phi(\vecz/|\vecz|)+
\frac{\hbar}{4r}\Big[4 \Delta_{\vect x \overline{\vect x}}-\mathcal R^2-(2n-2)\mathcal R\nonumber\\ 
& \hspace{1cm}-\frac{1}{4} (2n-1)(2n-3)\Big] \phi(\vect z/|\vect z|)
+\mathrm O(\hbar^2),\label{asymptotic Berezin transform}
\end{align}
where $\Delta_{\vect x \overline{\vect x}}=\sum_{j=1}^n \partial_{x_j \overline x_j}$ and $\mathcal R:=\sum_{j=1}^n \left( z_j \partial_{z_j}+ \overline z_j \partial_{\overline z_j}\right)$ denote the Laplace operator  and the radial derivative on $\mathbb R^{2n}\cong\mathbb C^n$, respectively.
\end{theorem}

\begin{proof}

Let us first note that it is enough to show that (\ref{asymptotic Berezin transform}) holds for all points $\vect z$ of the form 
$\vect z=r \hat{\vect e}_1$, with $r>0$, and $\hat{\vect e}_1 =(1,0,\ldots,0)\in \mathbb R^{2n}$.

Indeed, let $U\in \mathrm{SU}(n)$ and $\mathrm T_U$ be the unitary transformation defined in Proposition \ref{extended covariant symbol T_U}. Since the Laplace operator $\Delta_{\vect x \overline{\vect x}}$ as well as the radial derivative $\mathcal R$ are clearly invariant under unitary transformation $\mathrm T_U$, the right hand side of Eq. (\ref{asymptotic Berezin transform})  is likewise invariant under $\mathrm T_U$. 

By (2) of Proposition \ref{extended covariant symbol T_U}, the validity of Eq. (\ref{asymptotic Berezin transform}) for $\phi$ at $\vecz$ is therefore equivalent to its validity for $\mathrm T_{U^{-1}}\phi=\phi \circ U$ at $U^{-1} \vecz$. 

%
Since any given $\vecz$  can be mapped by a suitable $U\in \mathrm{SU}(n)$ into a point of the form $r \hat{\vect e}_1$
, with $r=|\vecz|$, it is indeed enough to prove (\ref{asymptotic Berezin transform}) only for points $\vecz$ of the latter form.


For the rest of the proof of Theorem \ref{theorem asymptotic Berezin transform}, we thus assume that $\vecz=r \hat{\vect e}_1$ with $r>0$.

Since $\K{0}(\vecx,\vecz)=e^{\vecx\cdot\vecz/\hbar}$ (see Eq. (\ref{coherent states}), with $p=0$). From Definition \ref{definition Berezin transform} and the norm estimate of the function $\K{0}(\cdot,\vect z)$ (see Eq. (\ref{norm coherent states})) we obtain
\begin{align}
\Bemap{0}(\phi)(\vect z)& =\mathfrak B_{\hbar,0} \To_\phi (\vecz) \nonumber\\ & =\frac{2\sqrt \pi}{\Gamma(n)} \left(\frac{r}{\hbar}\right)^{n-\frac{1}{2}}\int_{\Sn} \mathrm{exp}\left(\frac{2r}{\hbar} \big[\Re(x_1)-1\big]\right) \phi(\vecx) (1 + \mathrm O(\hbar))\ds(\vect x)\label{Berezin approximation 1}
\end{align}

Writing the coordinates $x_j$ of a point $\vect x \in \Sn$ in the form $x_j = y_j + \imath y_{n+j}$, with $y_j$ and $y_{n+j}$ real numbers. Note that the vector $\vect y = (y_1,\ldots,y_{2n}) \in S^{2n-1}=\{\mathbf a \in \mathbb R^{2n}\;|\; |\mathbf a|=1\}$. 

In order to estimate (\ref{asymptotic Berezin transform}), we define
\begin{align}
\mathcal A \Psi& = \left(\frac{r}{\pi\hbar}\right)^{n-\frac{1}{2}} \int_{S^{2n-1}}  \Psi(\boldsymbol \omega) \mathrm{exp}\left(\frac{\imath}{\hbar}f(\boldsymbol \omega)\right)  \mathrm d\Omega_{2n-1}(\boldsymbol \omega) \label{Auxiliar A-1}
\end{align}
where $\mathrm d \Omega_{2n-1}$ is the normalized surface measure on $S^{2n-1}$ and $f(\boldsymbol \omega)=-2\imath r\big[\omega_1-1 \big]$.

Let us introduce spherical coordinates for the variables $(\omega_1,\ldots, \omega_{2n})\in S^{2n-1}$:
\begin{align*}
\omega_1 & = \sin(\theta_{2n-1}) \cdots \sin(\theta_2)\cos(\theta_1)\;,\\
\omega_2 & = \sin(\theta_{2n-1}) \cdots \sin(\theta_2)\sin(\theta_1)\;,\\
&\hspace{0.2cm}\vdots\\
\omega_{2n-1}&=\sin(\theta_{2n-1})\cos(\theta_{2n-2})\;,\\
\omega_{2n}&= \cos(\theta_{2n-1})
\end{align*}
with $\theta_1\in(-\pi,\pi)$, $\theta_2,\theta_3,\cdots,\theta_{2n-1}\in(0,\pi)$.  

The function $f$ appearing in the argument of the exponential function in Eq. (\ref{Auxiliar A-1}) has a non-negative imaginary part and has only one critical point  (as a function of the angles) $\boldsymbol \theta_0$ which contributes to the asymptotic given by $\theta_1=0$, $\theta_j=\pi/2$, $j=2,\ldots,2n-1$. In addition, since
\begin{equation*}
\left.\frac{\partial^2 f}{\partial \theta_\ell\partial \theta_j}\right|_{\vtheta=\boldsymbol \theta_0}=2\imath r \delta_{j\ell}\;,
\end{equation*}
with $\delta_{ij}$ denoting the Kronecker symbol, then the Hessian matrix of $f$ evaluated at the critical point $\vtheta_0$ is equal to $f''(\vtheta_0)=2\imath r \mathbf I_{2n-1}$, where $\mathbf I_{s}$ denotes the identity matrix of size $s$. Moreover, $\mathrm{det}(f''(\vtheta_0))=(2\imath r)^{2n-1}$.
From the stationary phase method (see Ref. \cite{H90}) we obtain that

\begin{equation}\label{tfe}
\mathcal A \Psi= \frac{1}{N} \left[\sum_{\ell<k} \hbar^\ell \mathbf M_\ell \Psi\Big|_{cp} + \mathrm O(\hbar^k)\right],
\end{equation}
where $N$ is a  normalization constant such that $\int_{\boldsymbol \omega \in S^{2n-1}} \mathrm d \Omega_{2n-1} (\boldsymbol \omega) = 1$ and
\begin{align}
\mathbf M_\ell \Psi\Big|_{cp} & = \sum_{s=\ell}^{3\ell} \frac{\imath^{-\ell} 2^{-s}}{ s!(s-\ell)!} \left[ \left(-(f''(\boldsymbol \theta_0))^{-1}\right) \hat D \cdot \hat D\right]^s \Big [ \Psi(\boldsymbol \omega(\boldsymbol{\theta}))\nonumber\\
& \hspace{1cm}(\sin \theta_{2n-1})^{2n-2}\cdots \sin \theta_2 (\mathfrak p_{cp})^{s-\ell}\Big]\biggl|_{cp}\label{Mj}
\end{align}
with $\mathfrak g\big|_{cp}$ denoting the evaluation  at the critical point $\boldsymbol \theta_0$ of a given function $\mathfrak g$, 
\begin{align}
\mathfrak p_{cp}=\mathfrak p_{cp}(\boldsymbol \theta) & = -2\imath r \left(\sin \theta_{2n-1} \cdots \sin \theta_2 \cos \theta_1 - 1\right)-\imath r \left(\theta_1^2 +\sum_{j=2}^{2n-1} \left(\theta_{j}-\frac{\pi}{2}\right)^2\right),
\label{definition f|cp n=3}
\end{align}
and $\hat D$ the column vector of size $2n-1$ whose $j$ entry is $\partial_{\theta_j}$ (i.e. $(\hat D)_{j}=\partial_{\theta_j}$).

Since $\left(-(f''(\vtheta_0))^{-1}\right) \hat D \cdot \hat D= 
 \frac{\imath}{2r}\sum_{j=1}^{2n-1}\partial_{\theta_j \theta_j}$ and
$\partial_{\vtheta}^{\vect k} \mathfrak p_{cp}(\vtheta_0)=0$ for any multi index $\vect k \in \mathbb Z_+^{2n-1}$ that satisfies $|\vect k|\le 3$ or $|\vect k|$ is odd or $\theta_j$ is odd for some $j=1,\ldots, 2n-1$,
we obtain from Eq. (\ref{Mj}) that

\begin{align}
\mathbf M_0\Psi \Big|_{cp} 
& =\Psi(\hat{\mathbf e}_1), \label{M0 n=3}\\
\mathbf M_1 \Psi\Big|_{cp}
& =\frac{1}{4r}\Big[\Delta_{\vect y {\vect y}}-\partial_{y_1y_1}-(2n-1)\partial_{y_1}-\frac{1}{4} (2n-1)(2n-3)\Big] \Psi(\hat{\mathbf e}_1),\label{M1 n=3}
\end{align}
where we have used the chain rule to obtain Eq. (\ref{M1 n=3}).

From Eqs. (\ref{tfe}), (\ref{M0 n=3}) and (\ref{M1 n=3})
\begin{align}
\mathcal A \Psi& = \frac{1}{N}\biggl[1 + \frac{\hbar}{4r}\Big(\Delta_{\vect y {\vect y}}-\partial_{y_1y_1}-(2n-1)\partial_{y_1}\nonumber\\
& \hspace{2.5cm}-\frac{1}{4} (2n-1)(2n-3)\Big) \Psi(\hat{\mathbf e}_1)\biggl]+  \mathrm O(\hbar^2).\label{auxiliar mathcal A}
\end{align}

Thus, if we consider $\Psi(\Re(\vecx),\Im(\vecx))=\phi(\vecx)$,  $\vecx \in \Sn$, in Eq. (\ref{Auxiliar A-1}), 
then from Eqs. (\ref{Berezin approximation 1}), (\ref{auxiliar mathcal A}), the chain rule and the fact that $N=2\pi^n/ \Gamma(n)$ we conclude the proof of this theorem. 
%
%
%
\end{proof}


In subsection \ref{asymptotic K{-1}} we obtained  an asymptotic expression of the functions $\K{-1}(\cdot, \vect z)$, $\vect z \in \mathbb C^n$,
this result will allow us to obtain an asymptotic expansion for the Berezin transform $\Bemap{-1}$ similar to that obtained in Theorem \ref{theorem asymptotic Berezin transform}.

\begin{theorem}\label{theorem asymptotic Berezin transform p=1}
Let $n\ge 2$, $\vect z \in \mathbb C^n-\{\mathbf 0\}$ and $\phi$ be a smooth function on $\Sn$. Then for $\hbar \searrow 0$
\begin{align}
\Bemap{-1}(\phi)(\vect z)
=\phi(\vecz/|\vecz|)+ \frac{\hbar}{4r}\Big[4 \Delta_{\vect x \overline{\vect x}}-\mathcal R^2-(2n-2)\mathcal R\Big] \phi(\vect z/|\vect z|)+\mathrm O(\hbar^2),\label{asymptotic Berezin transform p=1}
\end{align}
where $\Delta_{\vect x \overline{\vect x}}=\sum_{j=1}^n \partial_{x_j \overline x_j}$ and $\mathcal R:=\sum_{j=1}^n \left( z_j \partial_{z_j}+ \overline z_j \partial_{\overline z_j}\right)$ denote the Laplace operator  and the radial derivative on $\mathbb R^{2n}\cong\mathbb C^n$, respectively.
\end{theorem}
\begin{proof}
By the same argument that was used in the proof of Theorem \ref{theorem asymptotic Berezin transform},  we only need to show this theorem for points $\vect z$ in the form $(r,0,\ldots, 0)$, with $r>0$. For the rest of the proof of Theorem \ref{theorem asymptotic Berezin transform p=1}, we thus assume that $\vecz=(r,0,\ldots,0)$  with $r>0$. 	
	
From Definition \ref{definition Berezin transform}
\begin{eqnarray}\label{integral expression Bemap_{-1}}
\Bemap{-1}(\phi)(\vect z) = \int_{\Sn} \phi(\vecx) \frac{|\Ku (\vecx,\vecz)|^2}{||\Ku(\cdot,\vecz)||^2_{\Sn}} \ds(\vecx).
\end{eqnarray}

Let $C$ the constant mentioned in Lemma \ref{asymptotic p=1} taken greater than one, and $W_\vect z$, $V_\vect z$ be the regions defined in Eq. (\ref{regions W_z, V_z}).

The integral in Eq. (\ref{integral expression Bemap_{-1}}) can be written as an integral over the region $W_{\vect z}$ plus an
integral over $V_{\vect z}$ denoted by the letters $\mathbf J_W(\vect z)$ and $\mathbf J_V (\vect z)$ respectively.


The integral over $W_{\vect z}$ is the one that gives us the main asymptotic term and the integral over $V_{\vect z}$ is actually $\mathrm O(\hbar^\infty)$.

The assertion $\mathbf J_{V}(\vect z)=\mathrm O(\hbar ^\infty)$ follows from the following inequality
\begin{equation*}
\frac{|\Ku(\vecx,\vecz)|^2}{||\Ku(\cdot,\vecz)||^2_{\Sn}} \le C_1 \left(\frac{|\vecz|}{\hbar}\right)^{n-\frac{3}{2}} e^{2\mu\frac{|\vecz|}{\hbar}}\left[\frac{|\vecz|}{\hbar}+1\right]^2(1+\mathrm O(\hbar)), \quad \vect x \in V_\vect z
\end{equation*}
which in turn is a consequence from Eq. (\ref{g_a^{s} evaluated in V_z}) and the norm estimate for the function  $\K{-1}(\cdot,\vect z)$ (see Eq. (\ref{norm coherent states})).

Let us now study the term $\mathbf J_W(\vect z)$. From Eq. (\ref{asymptotic function K{-1}}) 
\begin{align}
\mathbf J_W(\vect z) & =\frac{e^{2r/\hbar}}{||\Ku(\cdot,\vect z)||^2_{\Sn}}\frac{r}{\hbar(n-1)}\int_{W_\vect z}  \phi(\vecx) \mathrm{exp}\left(\frac{2r}{\hbar}\left[\Re(x_1)-1\right]\right)\nonumber\\[0.2cm]
&\hspace{3.5cm}\times 
\left[|x_1|+\hbar \frac{(n-\frac{5}{4}) \Re(x_1)}{r|x_1|}+\mathbf E(\vect z,\hbar)\right]\ds (\vect x), \label{A p=1}
\end{align}
where the error term $\mathbf E(\vect z,\hbar)$ is $\mathrm O(\hbar^2)$
uniformly with respect to $\vect x \in W_\vect z$, because in such a region we have $ 1/| \vect x\cdot \vect z| \le C/|\vect z|$.

Furthermore, for each $\vect x \in V_\vect z$, $\Re(x_1)-1= \frac{\Re(\vecx\cdot\vecz)}{|\vecz|}-1<\frac{1}{C}-1<0$.  Thus we can take the integral defining $\mathbf J_W(\vect z)$ over the whole sphere with an  error $\mathrm O(\hbar^\infty)$.

Then, from Eqs. (\ref{A p=1}) and (\ref{Auxiliar A-1})
\begin{equation*}
\mathbf J_W(\vect z)= \frac{e^{2r/\hbar}}{||\Ku(\cdot,\vect z)||^2_{\Sn}}\frac{r}{\hbar(n-1)}\left(\frac{r}{\pi \hbar}\right)^{\frac{1}{2}-n} \left[\mathcal A \Psi_1 + \hbar\frac{n-\frac{5}{4}}{r} \mathcal A \Psi_2 + \mathrm O(\hbar^2)\right],
\end{equation*}
where
\begin{equation*}
\begin{aligned}
\Psi_1(\Re(\vecx),\Im(\vecx))& =\phi(\vecx)|x_1|, \\
\Psi_2(\Re(\vect x), \Im (\vect x)) & = \phi(\vect x) \frac{\Re(x_1)}{|x_1|}
\end{aligned} 
\hspace{1.5cm} \vecx \in \Sn.
\end{equation*}

Following a procedure analogous to that performed in Theorem \ref{theorem asymptotic Berezin transform}, we obtain from Eq.(\ref{auxiliar mathcal A})
\begin{align*}
\mathbf J_W(\vect z)& = \frac{e^{2r/\hbar}}{||\Ku(\cdot,\vect z)||^2_{\Sn}}\frac{r}{\hbar(n-1)}\left(\frac{r}{\pi \hbar}\right)^{\frac{1}{2}-n} \frac{1}{N}\biggl[\Psi(\hat{\mathbf e}_1)+ \frac{\hbar}{4r}\biggl(\Delta_{\vect y \vect y}-\partial_{y_1 y_1}\\
&-(2n-1)\partial_{y_1} - \left(n-\frac{5}{2}\right)\left(n-\frac{3}{4}\right)\biggl) \Psi(\hat{\mathbf e}_1)+\mathrm O(\hbar^2)\biggl].
\end{align*}
Finally, from the norm estimate for the function $\Ku(\cdot,\vect z)$ (see Eq. (\ref{norm coherent states})), the chain rule and the fact that $N=2\pi^n/ \Gamma(n)$ we conclude the proof of this theorem.
\end{proof}

\section{Egorov-type theorem}\label{section Egorov-type theorem}


The covariant symbol (see Eq. (\ref{definition Berezin transform})) is not only defined for Toeplitz operators, but for any bounded linear operator. 
Hence it is natural to ask if there exists an analogue of the asymptotic expansion obtained in Theorems  \ref{theorem asymptotic Berezin transform} and \ref{theorem asymptotic Berezin transform p=1} for a more general operator than a Toeplitz operator. For this reason, in this section we prove an Egorov-type theorem which states that, for a pseudo-differential operator $A_\hbar$ acting on $L^2(\Sn)$ with principal symbol $\wp(A_\hbar)$ (see Appendix \ref{appendix pseudo-differential operators} where we give a brief description of what we mean by semiclassical pseudo-differential operators on a manifold), the covariant symbol of $A_\hbar$ goes to the composition of $\wp(A_\hbar)$ with the map $\vect z \mapsto (\vect z/|\vect z|, - \imath \vect z)$ when goes $\hbar$ to zero. Namely,


\begin{theorem}\label{egorov theorem}
	Let $p=0,-1$ and $A_\hbar$ be a semiclassical pseudo-differential operator of
	order zero with domain in $L^2(\Sn)$, $n \ge 2$. Then for $\vect z \in \mathbb C^n-\{\vect 0\}$ and $\hbar \searrow 0$ we have
\begin{eqnarray}\label{covariant symbol pseudo-differencial operator}
	\frac{\left\langle  A_\hbar \K{p}(\cdot,\vect z)
		,\K{p}(\cdot,\vect z)\right\rangle_{\Sn}}{\left\langle \K{p}(\cdot,\vect z),\K{p}(\cdot,\vect z) \right\rangle_{\Sn}} = \wp(A_\hbar) (\vect z/|\vect z|,-\imath \vect z) + \mathrm O(\hbar)\;,
\end{eqnarray}
where $\wp(A_\hbar)$ is the principal symbol of the semiclassical pseudo-differential operator $A_\hbar$.
\end{theorem}
\begin{proof}
	
Let $U\in \mathrm{SU}(n)$ and $\mathrm T_U$ be the unitary transformation defined in Proposition \ref{extended covariant symbol T_U}. Let us first note that by Propositions \ref{extended covariant symbol T_U} and \ref{T_U A_hbar T_U pseudo-differential operator}, the validity of Eq. (\ref{covariant symbol pseudo-differencial operator}) for $A_\hbar$ at $\vecz$ is equivalent to its validity for $\mathrm T_{U^{-1}}A_\hbar \mathrm T_U$ at $U^{-1} \vect z$.
Since any given $\vect z$  can be mapped by a suitable $U\in \mathrm{SU}(n)$ into a point of the form $\vect z= \lambda \hat{\vect u}_1$ with $\lambda= |\vect z|$ and $\hat{\vect u}_1=(1,0,\ldots,0)$ a canonical unit vector in $\mathbb R^{n}$, it is indeed enough to prove (\ref{covariant symbol pseudo-differencial operator}) only for points $\vect z$ of the latter form.

For the rest of the proof of Theorem \ref{egorov theorem}, we thus assume that $\vecz=\lambda \hat{\vect u}_1$  with $\lambda>0$.

Let us identify $\Sn$ with $S^{2n-1}$ via the map $\Upsilon$, which sends $\vect x=(x_1,\ldots,x_n) \in \Sn$ 
 to $(\Re(x_1),\Im(x_1),\ldots,\Re(x_n),\Im(x_n)) \in S^{2n-1}$.
 

Let us consider the following particular charts $(U_{\mathfrak c},\kappa_{\mathfrak c}\circ \Upsilon)$ of $\Sn$, with $U_{\mathfrak c}=\Upsilon^{-1}\circ \kappa_{\mathfrak c}^{-1} \mathfrak A$, $\mathfrak c=1,2,\ldots,2n-1$, where
\begin{equation*}
\mathfrak A=\left\{(\theta,v_3,\ldots,v_{2n})\in \mathbb R^{2n-1}\;|\;-\pi<\theta<\pi,\;\;\sum_{q=3}^{2n}v_q^2< 1-\frac{1}{8(n-1)}\right\}
\end{equation*}
and $\kappa_{\mathfrak c}^{-1}: \mathfrak A \to \Upsilon (U_{\mathfrak c})$ are defined by
\begin{equation*}
\begin{aligned}
\kappa_1^{-1}(\theta,\mathbf v)& =(r\cos(\theta),r\;\mathrm{sen}(\theta),\mathbf v)\;,\\
\kappa_\mathfrak a^{-1}(\theta,\mathbf v) & = \mathcal R_\mathfrak a \kappa_1^{-1}(\theta,\mathbf v)\;, 
\end{aligned} 
\hspace{1cm}r=\left(1-\sum_{q=3}^{2n} v_q^2\right)^{\frac{1}{2}},
\end{equation*}
with $\mathbf v=(v_3,\ldots,v_{2n})$, $\mathfrak a=2,\ldots,2n-1$ and  $\mathcal R_\mathfrak a \in \mathrm{SO}(2n)$ denoting the matrix which rotates the $(2n-1)$-sphere 
on the $(x_1,x_2)$ plane by $\pi$ and then followed by another rotation on the $(x_2,x_{\mathfrak a})$ plane by $\pi/2$. More specifically, $\mathcal R_{\mathfrak a}$ is given by the matrix
\begin{equation}\label{definition R_a}
\mathcal R_\mathfrak a= 
(-\hat{\mathbf e}_1,\hat{\mathbf e}_{\mathfrak a+1},
\hat{\mathbf e}_3, \ldots, \hat{\mathbf e}_{\mathfrak a},\hat{\mathbf e}_2,\hat{\mathbf e}_{\mathfrak a+2},\ldots,\hat{\mathbf e}_{2n})
\end{equation}
where $\{\hat {\mathbf e}_1,\ldots, \hat {\mathbf e}_{2n}\}$ is the canonical basis of $\mathbb R^{2n}$ regarding $\hat {\mathbf e}_{\mathfrak a}$, $\mathfrak a=1,\ldots 2n$, as a column vectors.
Note that $\hat{\vect u}_1=\Upsilon^{-1}(\hat{\mathbf e}_1)\notin U_{\mathfrak a} $  for $\mathfrak a=2,\ldots ,2n-1$, and that $\Sn=\bigcup_{\mathfrak c=1}^{2n-1} U_{\mathfrak c}$ (see Appendix \ref{complex sphere union charts} for details). 

Let $\{\mathfrak t_\mathfrak c\}$ be a partition of the unity associated to the open cover $\{U_{\mathfrak c}\}$. 
Since $\hat{\mathbf u}_1$ only belongs to $U_{1}$ then $\mathfrak t_1(\hat{\vect u}_1)=1$. Consider a set of functions $\{\varrho_\mathfrak c\}$  with $\varrho_\mathfrak c \in C_0^\infty(U_{\mathfrak c})$ such that 
$\mathfrak t_\mathfrak c \varrho_\mathfrak c=\mathfrak t_\mathfrak c$, for $\mathfrak c=1,\ldots,2n-1$. Then
\begin{align}
A_\hbar & = \sum_{\mathfrak c=1}^{2n-1}\left[ \mathfrak t_\mathfrak c A_\hbar \varrho_\mathfrak c + \mathfrak t_\mathfrak c A_\hbar (1-\varrho_\mathfrak c)\right]\;. \label{A_hbar sum partition}
\end{align}

Moreover, since $\mathrm{supp}(\mathfrak t_\mathfrak c)\cap \mathrm{supp}(1-\varrho_\mathfrak c)=\varnothing$ and condition $(\mathbf B)$ in the definition of a semiclassical pseudo-differential operator (see Eq (\ref{insignificante})) we have
\begin{equation}\label{def Ahbar}
\frac{\left\langle\mathfrak t_\mathfrak c A_\hbar (1-\varrho_\mathfrak c) \K{p}(\cdot,\vect z),\K{p}(\cdot,\vect z) \right\rangle_{\Sn}}{\|\K{p}(\cdot,\vect z)\|^2_{\Sn}}
= \mathrm O(\hbar^\infty)\;,\; \mathfrak c=1,\ldots,2n-1\;.
\end{equation}


The basic idea of the proof is to show that the main contribution to the left hand side of Eq. (\ref{covariant symbol pseudo-differencial operator}) comes from the term $\frac{\langle \mathfrak t_1 A_\hbar \varrho_1 \K{p}(\cdot,\vect z), \K{p}(\cdot,\vect z) \rangle_{\Sn}} {||\K{p}(\cdot, \vect z)||^2_{\Sn}}$. The stationary phase method will be used to obtain such a main contribution.

From the definition of a semiclassical pseudo-differential operator given in Appendix \ref{appendix pseudo-differential operators}, there exist $m\in \mathbb R$, $k \in \mathbb Z_+$ larger than $m+n$ and symbols $a_{\kappa_\mathfrak c}\in \mathrm S_{4n-2} (\langle p_\theta,p_\vect v\rangle^m)$ such that
\begin{align}
& \frac{\big\langle \mathfrak t_\mathfrak c A_\hbar \varrho_\mathfrak c\K{p}(\cdot,\vect z),\K{p}(\cdot,\vect z)\big	\rangle_{\Sn}}{\|\K{p}(\cdot,\vect z)\|^2_{\Sn}}  =
 \frac{(2\pi\hbar)^{-2n+1}}{\|\K{p}(\cdot,\vect z)\|^2_{\Sn}} \int\limits_{(\theta,\vect v)\in 
	\mathfrak A}\int\limits_{(\tilde\theta,\tilde {\vect v})
	\in \mathfrak A} \int\limits_{(p_\theta,p_\vect v)\in \mathbb R^{2n-1}}\nonumber\\
&\hspace{.5cm}\frac{
	(\overline{\K{p}(\cdot,\vect z)}\mathfrak t_\mathfrak c)((\kappa_\mathfrak c\circ\Upsilon)^{-1}(\theta,\vect v))}{|\Sn|} \exp\left(\frac{\imath}{\hbar} [(\theta,\vect v)-(\tilde \theta,\tilde{\vect v})]\cdot (p_\theta,p_\vect v)
\right)\nonumber\\
& \hspace{.5cm} \left(\frac{1+\hbar \mathrm M}{1+p_\theta^2+p_\vect v^2}\right)^k
\biggl((a_{\kappa_\mathfrak c})_t(\theta,\vect v,\tilde{\theta},\tilde{\vect v},p_\theta,p_\vect v;\hbar)\left( \varrho_\mathfrak c  \K{p}(\cdot,\vect z)\right)((\kappa_\mathfrak c \circ \Upsilon)^{-1} (\tilde{\theta},\tilde{\vect v}))  \biggl)\nonumber\\
&\hspace{8.6cm} J(\theta,\vect v)\mathrm dp_\theta \mathrm dp_\vect v\mathrm d\tilde{\theta}
\mathrm d \tilde{\vect v} \mathrm d \theta\mathrm d \vect v,\label{eq13 te}
\end{align}
where we have used the coordinates $(\theta,\mathbf v)$ in $\mathfrak A$ and where $|\Sn|$ denotes the surface area of the complex $n$-sphere, the Jacobian $J(\theta,\vect v)=1$ and the operator $\mathrm M$ is defined by
\begin{equation}\label{def op M}
\mathrm M= \frac{1}{\imath} \left(p_\theta \frac{\partial}{\partial \tilde\theta}+ \sum_{j=3}^{2n} p_{v_j} \frac{\partial}{\partial \tilde v_j}\right).
\end{equation}

Notice that on the right hand side of Eq. (\ref{eq13 te}) we have
\begin{equation}\label{eq9 te}
(1+\hbar \mathrm M)^k((a_{\kappa_\mathfrak c})_t \varrho_\mathfrak c \K{p}(\cdot,\vect z))= \sum_{s=0}^k\sum_{q=0}^s \binom{k}{s}\binom{s}{q} \hbar^s \mathrm M^{s-q}\left[ (a_{\kappa_\mathfrak c})_t \varrho_\mathfrak c\right] \mathrm M^{q} \K{p}(\cdot,\vect z)\;.
\end{equation}
For $q\ge 1$,  the action of the operator $\mathrm M^{q}$ on the function  $\K{p}(\cdot,\vect z)$ can be written as (see formula 0.430-2 of Ref. \cite{G94}))
\begin{align}\label{eq14 te}
\mathrm M^{q} \K{p}(\cdot,\vect z)(\kappa_\mathfrak c^{-1}(\tilde \theta,\tilde{\vect v}))= &  \mathrm M^{q} g\left(\frac{\overline{\vect z \cdot \tilde\vecx_\mathfrak c}}{\hbar}\right)= \sum_{d=1}^q F_{d,q}^\mathfrak c(\tilde\theta,\tilde {\vect v}) \frac{1}{\hbar^d}g^{(d)}\left(\frac{\overline{\vect z \cdot \tilde\vecx_\mathfrak c}}{\hbar}\right)
\end{align}
with $g=g_{\frac{1}{n-1}}$ the function defined in (\ref{def g cap3})  if $p=-1$ or $g$ the exponential function if $p=0$, $\;\tilde \vecx_\mathfrak c=\tilde \vecx_\mathfrak c(\tilde \theta,\tilde {\vect v})=(\kappa_\mathfrak  c\circ \Upsilon)^{-1}(\tilde\theta,\tilde{\vect v})$ and

\begin{equation}\label{def Fdq}
F_{d,q}^\mathfrak c(\tilde\theta,\tilde{\vect v})= \sum \frac{q!}{p_1! ... p_\ell!} \left(\frac{\mathrm M^{1}\overline{\vect z \cdot
		\tilde\vecx_\mathfrak c}}{1!}\right)^{p_1} \left(\frac{\mathrm M^{2}\overline{\vect z \cdot\tilde\vecx_\mathfrak c}}{2!}
\right)^{p_2}...\left(\frac{\mathrm M^{\ell}\overline{\vect z \cdot \tilde\vecx_\mathfrak c}}{\ell!}\right)^{p_\ell},
\end{equation}
where the summation stands for non-negative numbers $p_1,\ldots,p_\ell$ such that $p_1+2p_2+\ldots+\ell p_\ell=q$ and $d=p_1+p_2+\ldots+p_\ell$.

Let us first suppose that $p=-1$. Let $W_{\vect z}$, $V_{\vect z}$ be the regions defined in Eq. (\ref{regions W_z, V_z}). Notice that if $\tilde {\vect x}_\mathfrak c \in W_{\vect z}$, $\mathfrak c=1,\ldots,2n-1,$ then we can use the asymptotic expression of the derivative of any order of $g$ evaluated at $\tilde{\vect x}_\mathfrak c \cdot \vect z/\hbar$ (see Eq. (\ref{g_a^{s} evaluated in W_z})). For this reason, let us regard the integration with respect to the variable $(\tilde \theta, \tilde{\vect v})$ in Eq. (\ref{eq13 te}) over the regions  $\kappa_{\mathfrak c}\circ \Upsilon(W_\vect z \cap U_{\mathfrak c})$ and $\kappa_{\mathfrak c}\circ \Upsilon(V_\vect z\cap U_{\mathfrak c})$ separately and let us call them $\mathfrak I_{W,\mathfrak c}(\vect z)$ and $\mathfrak I_{V,\mathfrak c}(\vect z)$, respectively. Thus we have
\begin{equation*}
\frac{\big\langle \mathfrak t_\mathfrak c A_\hbar \varrho_\mathfrak c \K{-1}(\cdot,\vect z), \K{-1}(\cdot,\vect z)\big	\rangle_{S^n}}{\| \K{-1}(\cdot,\vect z)\|^2_{\Sn}}=\mathfrak I_{W,\mathfrak c}(\vect z)+ \mathfrak I_{V,\mathfrak c}(\vect z)\;.
\end{equation*}

Now we claim that $\mathfrak I_{V,\mathfrak c}(\vect z)= \mathrm O(\hbar^\infty)$. Let $(\tilde\theta,\tilde{\vect v}) \in \kappa_\mathfrak c\circ \Upsilon(V_\vect z \cap U_{	\mathfrak c})$ and $\tilde\vecx_\mathfrak c=(\kappa_\mathfrak c\circ \Upsilon)^{-1}(\tilde\theta,\tilde{\vect v})$. 

From definition of  the operator $\mathrm M$ (see Eq. (\ref{def op M})) we have  $\left|\mathrm M^{b} \tilde\vecx_\mathfrak c \cdot \vect z \right| \le (1+p_\theta^2+p_\vect v^2)^{\frac{b}{2}} \left|\mathrm N^{b} 
\;\tilde\vecx_\mathfrak c \cdot \vect z \right|$, where $\mathrm N= \frac{\partial}{\partial \tilde\theta}+\sum_{j=3}^{2n} \frac{\partial}{\partial \tilde v_j}$. Note that the action of the operator $\mathrm N$ on the function $\tilde{\vect x}_\mathfrak c \cdot\vect z$ involve inverse powers of the variable $\tilde r =(1-|\tilde{\vect v}|^2)^{1/2}$ which are bounded in the support of $\varrho_\mathfrak c$ and therefore do not create any singularities; in fact $\left|\mathrm N^{b} 
\;\tilde\vecx_\mathfrak c \cdot \vect z \right|\le C_1\tilde r^{-b}\le C_1(2\sqrt{2n-2})^b$ for some constant $C_1$. Then, from Eq. (\ref{def Fdq})
\begin{align}
\left|F_{d,q}^\mathfrak c(\tilde\theta,\tilde{\vect v})\right| & \le C_1\langle(p_\theta,p_\vect v)\rangle^{d} 
\;, \;\;\forall\;(\kappa_{\mathfrak c}\circ \Upsilon)^{-1}(\tilde \theta,\tilde{\vect v}) \in \mathrm{supp}(\varrho_\mathfrak c).\label{eq17 te}
\end{align}
On the other hand, since $a_{\kappa_\mathfrak c} \in S_{4n-2}(\langle(p_\theta, p_\vect v)\rangle^m)$ 
\begin{equation}\label{eq18 te}
\left|\mathrm M^{b}\left[(a_{\kappa_\mathfrak c})_t\varrho_\mathfrak c\right]\right| \le \langle(p_\theta,p_\vect v)\rangle^{b} 
\left|\mathrm N^{b}\left[(a_{\kappa_\mathfrak c})_t \varrho_\mathfrak c\right]\right| \le C_1 \langle(p_\theta,p_\vect v)\rangle^{b+m}.
\end{equation}

Then, from Eqs. (\ref{eq13 te}), (\ref{eq9 te}), (\ref{eq14 te}), (\ref{eq17 te}), (\ref{eq18 te}), Proposition \ref{proposition g_a^{s} evaluated} (specifically Eq. (\ref{g_a^{s} evaluated in V_z})) and  the estimate of the norm of $\K{-1}(\cdot,\vect z)$ (see Eq. (\ref{norm coherent states})) we have
\begin{align*}
\left|\mathfrak I_{V,\mathfrak c}(\vect z)\right| &\le  \frac{C_1 e^{\mu|\vect z|/\hbar} }{\hbar^{\frac{10n-3}{4}}\|\K{-1}(\cdot,\vect z)\|}\int\limits_{(\theta,\vect v)\in \mathfrak A} 
\int\limits_{(p_\theta,p_\vect v)\in \mathbb R^{2n-1}} \frac{\left|(\overline{\K{-1}(\cdot,\vect z)}\mathfrak t_\mathfrak c)((\kappa_\mathfrak c \circ \Upsilon)^{-1}(\theta,\vect v))\right|}
{(1+p_\theta^2+p_\vect v^2)^{\frac{k-m}{2}}}\\
&\hspace{8.8cm}\mathrm dp_\theta \mathrm dp_\vect v \mathrm d\theta \mathrm d\vect v.
\end{align*}
with $\mu<0$. Using the Cauchy-Schwartz inequality and that $n<k-m$ we conclude that $\mathfrak I_{V,\mathfrak c}(\vect z)=\mathrm{O}(\hbar^\infty)$.

Let us now study the terms $\mathfrak I_{W,\mathfrak c}(\vect z)$. From Eqs.   (\ref{eq13 te}), (\ref{eq9 te}), (\ref{eq14 te}), (\ref{def Fdq}), the estimate of the norm of $\Ku(\cdot,\vect z)$ (see Eq. (\ref{norm coherent states})) and Proposition \ref{proposition g_a^{s} evaluated} (specifically Eq. (\ref{g_a^{s} evaluated in W_z}) with $a=\frac{1}{n-1}$) we have
\begin{align}
\mathfrak I_{W,\mathfrak c}(\vect z) & =\frac{\hbar|\vect z|^{n-\frac{3}{2}}(n-1)e^{\frac{-2|\vect z|}{\hbar}}}{(\pi\hbar)^{n-\frac{1}{2}} (2\pi\hbar)^{2n-1}} \int\limits_{(\theta,\vect v)\in \kappa_\mathfrak c \circ \Upsilon(W_\vect z \cap U_{\mathfrak c})}\; \int\limits_{(\tilde\theta,\tilde {\vect v}) \in \kappa_\mathfrak c \circ \Upsilon(W_\vect z \cap U_{\mathfrak c})}
\int\limits_{(p_\theta,p_{\vect v})\in \mathbb R^{2n-1}}\nonumber\\
&\hspace{0.5cm} \left(\frac{\vect z \cdot \vecx_\mathfrak c}{\hbar(n-1)}\right)^{\frac{1}{2}} \mathfrak t_\mathfrak c(\vecx_\mathfrak c)
\exp\left(\frac{\imath}{\hbar} [(\theta,\vect v)-(\tilde \theta,\tilde{\vect v})]\cdot (p_\theta,p_\vect v)\right) \exp\left(\frac{\vect z \cdot \vecx_\mathfrak c}{\hbar}\right)\nonumber\\
&\hspace{0.5cm}\sum_{s=0}^k\sum_{q=0}^s \sum_{d=1}^q\binom{k}{s}\binom{s}{q} 
\frac{\hbar^{s-d}}{(1+p_\theta^2+p_\vect v^2)^{k}}\;\mathrm M^{s-q}\left[ (a_{\kappa_\mathfrak c})_t 
\varrho_\mathfrak c(\tilde{\vect x}_\mathfrak c)\right] F_{d,q}^\mathfrak c(\tilde\theta,\tilde {\vect v}) \nonumber\\
&\hspace{0.5cm}\left(\frac{\overline{\vect z \cdot 	\tilde\vecx_\mathfrak c}}{\hbar(n-1)}\right)^{\frac{1}{2}} \exp\left(\frac{\overline{\vect z \cdot \tilde\vecx_\mathfrak c} }{\hbar}\right) \big[1+\mathrm O(\hbar)\big]
\mathrm dp_\theta \mathrm dp_\vect v\mathrm d\tilde\theta \mathrm d\tilde{\vect v}\mathrm d\theta \mathrm d\vect v\;.\nonumber
\end{align}
with  $\tilde \vecx_\mathfrak c=\tilde \vecx_\mathfrak c(\tilde \theta,\tilde {\vect v})=(\kappa_\mathfrak c\circ \Upsilon)^{-1}(\tilde\theta,\tilde{\vect v})$, $\vecx_\mathfrak c=\vecx_\mathfrak c( \theta, {\vect v})=(\kappa_\mathfrak  c\circ \Upsilon)^{-1}(\theta,{\vect v})$. 

Then, as $\vect z=\lambda \hat{\mathbf u}_1$ we have 
\begin{align}
\mathfrak I_{W,\mathfrak c} (\vect z)& = \frac{\lambda^{n-\frac{3}{2}} }{(2\pi\hbar)^{2n-1}(\pi\hbar)^{n-\frac{1}{2}}} \int\limits_{(\theta,\vect v)\in \kappa_\mathfrak c \circ \Upsilon(W_\vect z \cap U_{\mathfrak c})}\int\limits_{(\tilde\theta,\tilde {\vect v}) \in \kappa_\mathfrak c \circ \Upsilon(W_\vect z \cap U_{\mathfrak c})}\int\limits_{(p_\theta,p_\vect v)\in \mathbb R^{2n-1}} \nonumber\\
& \hspace{0.5cm}
F_{d,q}^\mathfrak c(\tilde\theta,\tilde {\vect v})  \frac{\hbar^{s-d} \mathfrak t_\mathfrak c (\vect x_{\mathfrak c})
}{(1+p_\theta^2+p_\vect v^2)^{k}} 
\exp\left(\frac{\imath}{\hbar} \mathfrak f_\mathfrak c(\theta,\vect v,\tilde\theta,\tilde{\vect v},p_\theta,p_\vect v)\right) \left(\lambda\mathfrak f^{\mathfrak c}(\theta,\vect v)\right)^{\frac{1}{2}} \nonumber\\
&\hspace{0.5cm}\left(\lambda\mathfrak f^{\mathfrak c}(-\tilde \theta,-\tilde{\vect v})\right)^{\frac{1}{2}}
\sum_{s=0}^k\sum_{q=0}^s \sum_{d=1}^q \binom{k}{s}\binom{s}{q} 
\mathrm M^{s-q}\Big[ (a_{\kappa_\mathfrak c})_t (\theta,\vect v,\tilde\theta,\tilde{\vect v},p_\theta,p_\vect v)\nonumber\\
& \hspace{0.5cm}\varrho_\mathfrak c(\tilde{\vect x}_{\mathfrak c})\Big]\big[1+\mathrm O(\hbar)\big] \mathrm dp_\theta \mathrm dp_\vect v\mathrm d\tilde\theta \mathrm d\tilde{\vect v}\mathrm d\theta \mathrm d\vect v\;,\label{eq19 te}
\end{align}
where
\begin{equation*}
\mathfrak f_\mathfrak c(\theta,\vect v,\tilde\theta,\tilde{\vect v},p_\theta,p_\vect v)= 2\imath \lambda +p_\theta(\theta-\tilde\theta) +p_\vect v(\vect v-\tilde{\vect v})-\imath\lambda [\mathfrak f^\mathfrak c(\theta,\vect v)+\mathfrak f^\mathfrak c(-\tilde \theta,-\tilde{\vect v})]
\end{equation*}
with $\mathfrak f^{1}(\theta,\vect v)=re^{\imath \theta}$ and $\mathfrak f^{\mathfrak c}(\theta,\vect v)=-r\cos \theta+ \imath v_{\mathfrak c+1}$ for $\mathfrak c=2,\ldots,2n-1$.

From Eqs. (\ref{eq19 te}), (\ref{eq17 te}), (\ref{eq18 te}) and using the arguments given above to study $\mathfrak I_{V,\mathfrak c}(\vect z)$, it can be show that $\mathfrak I_{W,\mathfrak c}(\vect z)=\mathrm O(\hbar^\infty)$ for $\mathfrak c=2,\ldots,n$.

Let us now study the term $\mathfrak I_{W,1}$. Since $a_{\kappa_1}\in \mathrm S_{4n-2}(\langle(p_\theta,p_\vect v)\rangle^m)$ is a classical symbol, there exist a sequence $(a_{\kappa}^{(j)})\subset \mathrm S_{4n-2}\left(\langle(p_\theta,p_\vect v)\rangle^m\right)$ (independent of $\hbar$) such that if $N>3n/2$ and $\Psi_{N}=\sum_{\ell=0}^N\hbar^\ell a_{\kappa}^{(\ell)}$, then for any $\mathbf b \in\mathbb Z_+^{2n-1}$, there exist $\hbar_{N,\mathbf b}>0$ and $C_{N,\mathbf b}$ such that
\begin{equation}
\left|\partial^{\mathbf b} \left(a_{\kappa_1}-\Psi_{N} \right)\right| \le C_{N,\mathbf b}\;\hbar^N \langle(p_\theta,p_\vect v)\rangle^m\label{eq20 te}
\end{equation}
uniformly in $\mathbb R^{4n-2}\times (0,\hbar_{N,\mathbf b})$.

Let us denote by $\mathfrak I_{W, 1}^1(\vect z) $ and $\mathfrak I_{W,1}^2(\vect z) $ the right side of Eq. (\ref{eq19 te}) 
with $a_{\kappa_1} - \Psi_{N} $ and $\Psi_{N} $ instead of $a_{\kappa_1} $ respectively. Note that $ \mathfrak I_{W, 1}(\vect z) =\mathfrak I_{W, 1}^1(\vect z) + \mathfrak I_{W, 1}^2 (\vect z)$. Furthermore, from Eqs. (\ref{eq19 te}), (\ref{eq17 te}), (\ref{eq18 te}) and (\ref {eq20 te}) 
it can be shown that $\lim_{\hbar \to 0}\mathfrak I_{W, 1}^1 (\vect z)=0$

Finally, let us use the stationary phase method  to study the term $\mathfrak I_{W,\mathfrak 1}^2(\vect z)$. 

The function $\mathfrak f_1$ has a non-negative imaginary part and one critical point $\boldsymbol\vartheta_0$ that contributes to the main asymptotic given by 
\begin{equation*}
\boldsymbol\vartheta_0:=(\theta_0,\vect v_0,\tilde \theta_0,\tilde{\vect v}_0,p_{\theta_0},p_{\vect v_0})=(0,\mathbf 0, 0,\mathbf 0,-\lambda,\mathbf 0).
\end{equation*}

Since  $\det(\mathfrak f_1''(\boldsymbol\vartheta_0)/2\pi \imath\hbar)=
\lambda^{2n-1}/2^{4n-2}(\pi\hbar)^{6n-3}$ (with $\mathfrak f_1''$ the Hessian matrix of $\mathfrak f_1$), $\mathfrak f_1(\boldsymbol{\vartheta_0})=0$, $\mathfrak t_1((\kappa_1\circ \Upsilon)^{-1} (\theta_0,\vect v_0)) =\mathfrak t_1(\hat{\mathbf u}_1)=1$ and $\varrho_1(\hat{\mathbf u}_1)=1$ , then we obtain from the stationary phase method (see Ref. \cite{H90}) 
\begin{align*}
\mathfrak I_{W,1}^2 (\vect z)& =  a_{\kappa_1}^{(0)}(0,\vect 0, -\lambda, \vect 0) + \mathrm O(\hbar)\\
& =\wp(A_\hbar)\left(\hat{\mathbf u}_1,-\lambda \imath \hat{\mathbf u}_1\right)+\mathrm O(\hbar)\;.
\end{align*}
where we have used that $F_{s,s}^1(\tilde\theta_0,\tilde {\vect v}_0)= (\lambda^2)^s$ and the definition of the principal symbol. Thus, we have proved Eq. (\ref{covariant symbol pseudo-differencial operator}) if $p=-1$.

The case $ p=0 $ is proved in the same way. 
\end{proof}

\appendix
\section{Pseudo-differential operators}\label{appendix pseudo-differential operators}

In this appendix, we give a brief description of what we mean by semiclassical pseudo-differential operators on a manifold. See \cite{M02} y \cite{Z-E} for details on definitions of semiclassical pseudo-differential operators on $\mathbb R^n$ and manifolds, respectively.
\begin{definition}
	For $\boldsymbol\eta \in \mathbb R^n$, let us define $\langle \boldsymbol\eta \rangle : =\sqrt{1+|\boldsymbol\eta|^2}$, where
	 $|\boldsymbol\eta|^2=\eta_1^2+\ldots +\eta_n^2$.
\end{definition}
\begin{definition}
	Let $\hbar_0 > 0$ and $m \in \mathbb R$. The space of symbols $\mathrm S_{2n}(\langle\boldsymbol\xi\rangle^m)$ denotes the set of functions $p=p(\mathbf b,\boldsymbol\xi;\hbar)$, $\mathbf b,\boldsymbol\xi 
	\in\mathbb R^n$, $\hbar\in(0,\hbar_0]$ which are smooth in the variable $\vecz=(\mathbf b,\boldsymbol\xi)$ and for any $\mathbf k \in  \mathbb Z_+^{2n}$
	\begin{equation*}
	\partial ^{\mathbf k}p(\mathbf b,\boldsymbol\xi;\hbar)=\mathrm O(\langle \boldsymbol\xi\rangle^m) 
	\end{equation*}
	uniformly with respect to $(\vecz,\hbar)  \in\mathbb R^{2n} \times(0,\hbar_0]$.
\end{definition}

\begin{definition}\label{simbolo clasico}
	Let $m \in\mathbb R$, $a \in \mathrm S_{2n}(\langle\boldsymbol\xi\rangle^m)$ and $(a^{(j)})$ a sequence of symbols of $\mathrm S_{2n}(\langle \boldsymbol \xi\rangle^m)$ independent of $\hbar$, with $a^{(0)}$ non identically zero. 
	Then we say that $a$ is asymptotically equivalent to the formal sum
	 $\sum_{j=0}^\infty \hbar^j a^{(j)}$ in $\mathrm S_{2n}(\langle \boldsymbol\xi 	\rangle^m)$ if and only if for any $N \in \mathbb N$ and for any $\mathbf k\in \mathbb Z_+^n$ there exist $\hbar_{N,\mathbf k}>0$ and  $C_{N,\mathbf k}>0$ such that
	\begin{equation*}
	\left|\partial^{\mathbf k}\left(a-\sum_{j=0}^N \hbar^j a^{(j)}\right)\right| \le C_{N,\mathbf k} \hbar^N \langle \boldsymbol\xi\rangle^m
	\end{equation*}
	uniformly on $\mathbb R^{2n} \times (0,\hbar_{N,\mathbf k}]$. In this case we say that $a$ is a classical symbol.
\end{definition}

Now we will give the definition of semiclassical pseudo-differential operator on $\mathbb R^n$. In order

\begin{definition}\label{def op pseudo-diferencial}
	Let $m \in \mathbb R$,  $a \in \mathrm S_{2n}(\langle \boldsymbol\xi \rangle^m)$ and  $t\in[0,1]$. The action of a semiclassical pseudo-differential operator in $\mathbb R^n$, denoting by $\Op(a)$, on a smooth function
	$f \in C_0^\infty(\mathbb R^n)$ is defined by
	\begin{align*}
	\Op(a)f(\mathbf b)=\frac{1}{(2\pi\hbar)^n} \int_{\mathbf c \in \mathbb R^n} \int_{\boldsymbol\xi \in \mathbb R^n} & e^{\imath(\mathbf b-\mathbf c)\cdot 
		\boldsymbol \xi/\hbar}\\
	&\big(\mathrm L(\boldsymbol\xi,\hbar \mathrm D_{\mathbf c})\big)^k\big(a_t(\mathbf b,\mathbf c,\boldsymbol
	\xi;\hbar)f(\mathbf c)\big) \mathrm d \mathbf c \mathrm d \boldsymbol\xi\;,
	\end{align*}
	where $a_t(\mathbf b,\mathbf c,\boldsymbol \xi;\hbar)=a((1-t)\mathbf b+t\mathbf c,\boldsymbol\xi;\hbar)$, $k$ is any non-negative integer number such that $m+n<k$ and
	\begin{equation*}
	\mathrm L(\boldsymbol\xi,\hbar \mathrm D_{\mathbf c})=\frac{1+\hbar \boldsymbol\xi \cdot \mathrm D_{\mathbf c}}{1+|\boldsymbol\xi|^2} \;,
	\hspace{1cm}\mathrm D_{\mathbf c}=\frac{1}{\imath }\left(\frac{\partial}{\partial c_1},\ldots,\frac{\partial}{\partial c_n}\right).
	\end{equation*}
	If $a$ is a classical symbol, the function $a^{(0)}$ will be called principal symbol of the semiclassical pseudo-differential operator  $\Op(a)$ and will be denoted by $\wp(\Op(a))$.
\end{definition}

\begin{remark}
	In the previous definition, the operator $\mathrm L(\boldsymbol\xi,\hbar \mathrm D_{\mathbf c})$ is introduced in order to make sense of the possibly undefined integral
\begin{equation*}
\int_{\mathbf b \in \mathbb R^n} \int_{\boldsymbol\xi \in\mathbb R^n} e^{\imath(\mathbf b-\mathbf c)\cdot \boldsymbol \xi/\hbar}a_t(\mathbf b,\mathbf c,\boldsymbol \xi;\hbar)f(\mathbf c) \mathrm d \mathbf c \mathrm d \boldsymbol\xi.
\end{equation*}
\end{remark}
\begin{remark}
	For the values $t=0$, $t=\frac{1}{2}$ and $t=1$, $\mathrm {Op}_\hbar^t(a)$  is called standard quantization or left quantization, Weyl quantization and right quantization respectively.
\end{remark}
\noindent\textbf{Notation.} Let $\mathrm M$ be a smooth $n$-dimensional manifold with atlas $\mathcal F$. The elements of $\mathcal F$ are called charts and will be denoted by $(U_\kappa,\kappa)$ with $\kappa:U_\kappa \to V_\kappa$ and $\kappa$ an homeomorphism between the open sets $U_\kappa \subset \mathrm M$ and $V_\kappa \subset \mathbb R^n$. 


Now we will give a brief description of the concept of semiclassical pseudo-differential operator on 
$\mathrm M$.

\begin{definition}
	Let $A_\hbar: C^\infty(\mathrm M) \to C^\infty(\mathrm M)$ be a linear operator. 
	 We say that $A_\hbar$ is a semiclassical pseudo-differential operator of order zero iff there exist $m \in\mathbb R$ and $0\le t \le 1 $ such that the following three condition $\mathrm{(\mathbf A)}$, $\mathrm{(\mathbf B)}$ and $\mathrm{(\mathbf C)}$ are satisfied:
	
	$\mathrm{(\mathbf A)}$ For each chart $(U_\kappa,\kappa)$, there exist a classical symbol $a_\kappa \in \mathrm S_{2n}(\langle \boldsymbol\xi\rangle^m )$, such that for all $u \in C^\infty(\mathrm M)$ we have
	\begin{equation}\label{eq25 te}
	\Phi A_\hbar (\Psi u)= \Phi \kappa^* \Op(a_\kappa)(\kappa^{-1})^* (\Psi u),\quad \forall\; \Phi,\Psi \in C_0^\infty(U_\kappa),
	\end{equation}
	where we are extending $(\kappa^{-1})^*(\Psi u)$ as zero outside of the set $V_\kappa$.
	
	$\mathrm{(\mathbf B)}$ If $\Phi_j\in C_0^\infty(\mathrm M)$, $j=1,2$, with disjoint supports, then
	\begin{equation}\label{insignificante}
	\| \Phi_1 \;A_\hbar \Phi_2 \Psi\|_{\Sn} \le \mathrm O(\hbar^\infty)\|\Psi\|_{\Sn},\quad\forall\; \Psi\in C^{\infty}(\mathrm M).
	\end{equation}
	
	$\mathrm{(\mathbf C)}$ Given two charts $(U_\kappa,\kappa)$, $(U_{\nu},\nu)$ such that $U_\kappa \cap U_\nu \ne \varnothing$ we require the condition
	\begin{equation*}
	a_\kappa^{(0)}(\mathbf b,\boldsymbol\xi)=a_\nu^{(0)}((\nu \circ \kappa^{-1})\mathbf b,(\nu \circ \kappa^{-1})^*\boldsymbol\xi)\;,
	\hspace{0.5cm}\forall(\mathbf b,\boldsymbol\xi) \in T^*(\kappa(U_\kappa \cap U_\nu))\;.
	\end{equation*}
	The last condition implies that we can define the principal symbol $\wp(A_\hbar)$ of the pseudo-differential operator $A_\hbar$ by
	\begin{equation}\label{eq26 te}
	\wp(A_\hbar)(\vecx,\boldsymbol\eta)=a_\kappa^{(0)}(\kappa \vecx,(\kappa^{-1})^*\boldsymbol\eta) \;,\hspace{0.5cm} \forall\; (\vecx,\boldsymbol\eta)\in T^*(\mathrm M)\;,
	\end{equation}
	where $(U_\kappa,\kappa)$ is any chart such that $\vect x \in U_\kappa$.
\end{definition}


Since $\Sn$ can be identified with $S^{2n-1}$ via the map $\Upsilon$, which sends $\vect x=(x_1,\ldots,x_n) \in \Sn$ to $(\Re(x_1), \Im(x_1), \ldots, \Re(x_n), \Im(x_n)) \in S^{2n-1}$, we can consider $\Sn$ as a  smooth $2n-1$ dimensional manifold.



Let $A_\hbar$ be a semiclassical pseudo-differential operators on $\Sn$. Based on relating corresponding properties of functions and $A_\hbar$ on a given chart and its rotated chart under the action of $\mathrm{SU}(n)$ on the complex $n$-sphere $\Sn$, we can show the relationship between the principal symbols of $A_\hbar$ and $\mathrm T_{U^{-1}}A_\hbar \mathrm T_U$ (with $U\in \mathrm{SU}(n)$ and $\mathrm T_U$ defined as in Proposition \ref{extended covariant symbol T_U}) as the following proposition establish it. This proposition will be useful to prove Theorem \ref{egorov theorem}.

\begin{prop}\label{T_U A_hbar T_U pseudo-differential operator}
	Let $U\in \mathrm{SU}(n)$  and $A_\hbar$ be a semiclassical pseudo-di\-ffe\-ren\-tial operator on $\Sn$. Then $\mathrm T_{U^{-1}} A_\hbar \mathrm T_U$ is a semiclassical pseudo-di\-ffe\-ren\-tial operator on $\Sn$ with principal symbol
	\begin{equation*}
	\wp(T_{U^{-1}}A_\hbar T_U)(\vecx,\boldsymbol\eta)= \wp(A_\hbar)(U\vecx,(U^{-1})^*\boldsymbol\eta) \;,\hspace{0.7cm} \forall \;(\vecx,\boldsymbol\eta)\in T^*(\Sn).
	\end{equation*}
\end{prop}

\begin{proof}
	Let $(\vect x,\boldsymbol\eta)\in T^*(\Sn)$  and $(X_\tau,\tau)$ be  any chart such that $\vecx \in X_\tau$. Let us denote by $(X_{\tilde\tau},\tilde\tau)$ the chart given by $X_{\tilde\tau}= U X_\tau$ and $\tilde\tau=\tau \circ U^{-1}$.
	
	Let $v \in C^\infty(\Sn)$, and $\phi,\psi \in C_0^\infty(X_\tau)$. Let us denote by $\tilde v =  \mathrm T_U v$, $\tilde \psi =  \mathrm T_U \psi$ and
	$\tilde \phi = \mathrm T_U \phi$. Note that $\tilde\psi,\tilde\phi \in C_0^\infty(X_{\tilde\tau})$. From Eq. (\ref{eq25 te})
	\begin{align*}
	(\psi \mathrm T_{U^{-1}}A_\hbar \mathrm T_U (\phi v))(\vect x) & = (\tilde \psi A_\hbar \tilde\phi \tilde v)(U\vecx)\\
	& = \left[\tilde \psi \;(\tilde\tau)^* \mathrm{Op}_\hbar^t(a_{\tilde\tau})(\tilde\tau^{-1})^*(\tilde\phi\tilde v)\right](U\vecx)\\
	& = \left[\psi\tau^* \mathrm{Op}_\hbar^t(a_{\tilde \tau})(\tau^{-1})^*(\phi v)\right](\vecx)\;.
	\end{align*}
	which establishes that condition $(\mathbf A)$ holds for the operator $\mathrm T_{U^{-1}}A_\hbar \mathrm T_U$. 
	
	Condition $(\mathbf B)$ follows from Eq. (\ref{insignificante}) and the $\mathrm{SU}(n)$-invariance of $\mathrm d\Sn$. 
	
	Let us now consider another chart $(X_{\tau'},\tau')$ such that $X_\tau \bigcap X_{\tau'}=\varnothing$. Condition $(\mathbf C)$ follows from the equalities $\tau(X_\tau)=\tilde \tau({X}_{\tilde \tau})$, $\tau(X_{\tau'})=\tilde {\tau'}(X_{\tilde {\tau'}})$ and $\tilde{\tau'} \circ (\tilde \tau)^{-1}=\tau' \circ \tau^{-1}$.
	
	Thus we conclude that $\mathrm T_{U^{-1}}A_\hbar \mathrm T_U$ is a pseudo-differential operator with principal symbol (see Eq. (\ref{eq26 te})). 
	\begin{align*}
	\wp(\mathrm T_{U^{-1}}A_\hbar \mathrm T_U) (\vecx,\boldsymbol\eta) & = a_{\tilde \tau}^{(0)}(\tau \vecx, (\tau^{-1})^*\boldsymbol\eta) \\
	& = 
	a_{\tilde\tau}^{(0)}(\tilde\tau \circ U\vecx,(\tilde\tau^{-1})^*(U^{-1})^* \boldsymbol\eta)\\
	& = \wp(A_\hbar)(U\vecx,(U^{-1})^*\boldsymbol\eta )\;.
	\end{align*}
\end{proof}

\section{The complex n-sphere $\Sn$ as the union of the charts $U_\mathfrak c$}\label{complex sphere union charts}

In this appendix we prove that $\Sn=\bigcup_{\mathfrak c=1}^{2n-1} U_{\mathfrak c}$ and that $\hat{\vect u}_1 \notin U_{\mathfrak a}$ for $\mathfrak a=2,\ldots,2n-1$, where $\hat{\vect u}_1=(1,0,\ldots,0)$ is a canonical vector in $\mathbb R^n$.

To prove that $\Sn=\bigcup_{\mathfrak c=1}^{2n-1} U_{\mathfrak c}$, let us first note that
\begin{equation*}\label{expresion Sn}
\Sn=\Upsilon^{-1}\left(\left\{(r\cos\theta,r\sin\theta,
\mathbf v)|-\pi<\theta\le \pi,|\mathbf v|\le 1,\mathbf v\in \mathbb R^{2n-2},r^2={1-|\mathbf v|^2}\right\} \right)
\end{equation*}
where $\Upsilon$ is the map that identify $\Sn$ with $S^{2n-1}$ (see its definition at the beginning of the proof of Theorem \ref{egorov theorem}).

Consider $\vecx \in \Sn$ and write it as $\vect x= \Upsilon^{-1}\big((r\cos \theta,r\sin \theta,\vect v)\big)$. Let us assume that $\vect x \notin U_1$. Then we have the following two cases:
\begin{enumerate}
	\item[i)] $\displaystyle 1 - \frac{1}{8(n-1)}\le |\vect v|^2\le 1$.
	\item[ii)] $\theta =\pi$ and $\displaystyle |\vect v|^2 < 1 - \frac{1}{8(n-1)}$.
\end{enumerate} 
\textbf{Case} i) Let $v_j=\max \{v_3,\ldots,v_{2n}\}$, then $v_j>\frac{1}{2\sqrt{n-1}}$, because otherwise $|\vect v|^2\le \frac{1}{2}< 1 - \frac{1}{8(n-1)}$.

Let us define $\tilde{\mathbf v}=(\tilde v_3,\ldots, \tilde v_{2n})\in \mathbb R^{2n-2}$ such that $\tilde v_i=v_i$, $i\ne j$ and $\tilde v_j=r\sin(\theta)$. Since
\begin{equation*}
\left(-\frac{r}{\tilde r}\cos(\theta)\right)^2+\left(\frac{v_j}{\tilde r}\right)^2=1,\quad \mbox{ where } \;\tilde r^2=1-|\tilde{\mathbf v}|^2
=r^2\cos^2(\theta)+v_j^2>0,
\end{equation*}
then there exist $-\pi\le \tilde \theta<\pi$  such that $\cos(\tilde \theta)=-\frac{r}{\tilde r}\cos(\theta)$ and $\sin(\tilde\theta)=
\frac{v_j}{\tilde r}$. Actually $\tilde \theta \ne \pi$ because otherwise $v_j$ would have to be equal to zero.

Since $|\tilde{\vect v}|^2=1-r^2\cos^2(\theta)-v_j^2<1-\frac{1}{8(n-1)}$ then $(\tilde \theta, \tilde{\vect v}) \in \mathfrak A$. Moreover, using the explicit expression of the matrix $\mathcal R_{j-1}$ (see Eq. (\ref{definition R_a})), one can check that $\vecx=(\kappa_{j+1}\circ \Upsilon)^{-1}(\tilde \theta,\tilde{\mathbf v})\in U_{j+1}$.

\textbf{Case} ii) Let us define  $v_j$ and the vector $\tilde{\vect v}$ as above. Note that $\tilde r^2=1-|\tilde{\vect v}|^2=r^2 + v_j^2\ge r^2=1-|\vect v|^2>0$. 

Consider $-\pi < \tilde \theta < \pi$ such that $\cos (\tilde \theta)= \frac{r}{\tilde r}$ and $\sin (\tilde \theta)=\frac{v_j}{\tilde r}$. Moreover, $|\tilde {\vect v}|^2=1-r^2-v_j^2\le |\vect v|^2< 1- \frac{1}{8(n-1)}$. Therefore, $(\tilde \theta,\tilde{\vect v}) \in \mathfrak A$ and one can check that $\vect x \in U_{j-1}$.

Let us now prove that $\hat{\vect u}_1$ only belongs to $U_{\mathfrak 1}$. Let us assume that $\hat{\vect u}_1 \in U_\mathfrak a$ for some $\mathfrak a=2,\ldots, 2n-1$, then there exist $(\theta,\vect v) \in \mathfrak A$ such that
\begin{equation*}
\hat{\vect u}_1=(\kappa_\mathfrak a\circ \Upsilon)^{-1}(\theta,\vect v)=\Upsilon^{-1}\big((-r\cos \theta,v_{\mathfrak a+1},v_3,\ldots,v_\mathfrak a,r\sin \theta,v_{\mathfrak a+2},\ldots,v_{2n})\big),
\end{equation*}
which is impossible since $-\pi < \theta <\pi$.

\end{document}